\documentclass[letter,10pt]{amsart}
\usepackage{amssymb,amsmath,amsfonts}
\usepackage{mathrsfs}
\usepackage[left=1 in, right=1 in,top=1 in, bottom=1 in]{geometry}

\newtheorem{theorem}{Theorem}[section]
\newtheorem{lemma}[theorem]{Lemma}

\newtheorem{corollary}[theorem]{Corollary}

\newtheorem{remark}[theorem]{Remark}

\makeatletter
\renewcommand \theequation {%
\ifnum \c@section>\z@ \@arabic\c@section.%
\fi\@arabic\c@equation} \@addtoreset{equation}{section}
\makeatother

\providecommand{\ud}[1]{\mathrm{d}{#1}}
\providecommand{\abs}[1]{\left\vert#1\right\vert}
\providecommand{\nm}[1]{\left\Vert#1\right\Vert}

\providecommand{\tm}[2]{\left\Vert#1\right\Vert_{L^2(#2)}}
\providecommand{\im}[2]{\left\Vert#1\right\Vert_{L^{\infty}(#2)}}

\providecommand{\lnnm}[1]{{\left\Vert#1\right\Vert}_{L^{\infty}L^{\infty}}}
\providecommand{\lnm}[1]{\left\Vert#1\right\Vert_{L^{\infty}}}

\providecommand{\tnnm}[1]{{\left\Vert#1\right\Vert}_{L^{2}L^2}}

\def\dt{\partial_t}
\def\p{\partial}
\def\ls{\lesssim}

\def\half{\frac{1}{2}}
\def\rt{\rightarrow}
\def\r{\mathbb{R}}
\def\no{\nonumber}
\def\ue{\mathrm{e}}

\def\u{U^{\e}}
\def\ub{\mathscr{U}^{\e}}
\def\bu{\bar U^{\e}}
\def\bub{\bar{\mathscr{U}}^{\e}}
\def\uc{U}
\def\ubc{\mathscr{U}}
\def\buc{\bar U}
\def\bubc{\bar{\mathscr{U}}}
\def\ls{\lesssim}

\def\half{\frac{1}{2}}
\def\e{\epsilon}
\def\s{\mathcal{S}}
\def\vx{\vec x}
\def\vw{\vec w}

\def\nx{\nabla_{x}}

\def\px{\p_{\eta}}

\def\rt{\rightarrow}
\def\l{\lambda}

\def\ll{\mathcal{L}}

\def\r{\mathbb{R}}
\def\f{f}
\def\k{\kappa}
\def\q{Q}
\def\qb{\mathscr{Q}}
\def\v{\mathscr{V}}
\def\ff{\mathcal{F}}

\begin{document}
\title{Diffusive Limit with Geometric Correction of
Unsteady Neutron Transport Equation}
\author{Lei Wu}
\address{
Department of Mathematical Sciences\\
Carnegie Mellon University \\
Wean Hall 6113, 5000 Forbes Avenue, Pittsburgh, PA 15213, USA } \email[L.
Wu]{lwu2@andrew,cmu.edu}
\subjclass[2000]{35L65, 82B40, 34E05}

\begin{abstract}
We consider the diffusive limit of an unsteady neutron transport
equation in a two-dimensional plate with one-speed velocity. We show the solution can be approximated by the sum of interior solution, initial layer, and boundary layer with geometric correction. Also, we construct a counterexample to the classical theory in \cite{Bensoussan.Lions.Papanicolaou1979} which states the behavior of solution near boundary can be described by the Knudsen layer derived from the Milne problem.\\
\textbf{Keywords:} compatibility condition, $\e$-Milne problem, Knudsen layer,
geometric correction.
\end{abstract}

\maketitle

\pagestyle{myheadings} \thispagestyle{plain} \markboth{LEI WU}{GEOMETRIC CORRECTION FOR DIFFUSIVE EXPANSION OF NEUTRON
TRANSPORT EQUATION}
%
%
%
%

\section{Introduction and Notation}

\subsection{Problem Formulation}

We consider a homogeneous isotropic unsteady neutron transport
equation in a two-dimensional unit plate $\Omega=\{\vx=(x_1,x_2):\
\abs{\vx}\leq 1\}$ with one-speed velocity $\Sigma=\{\vw=(w_1,w_2):\
\vw\in\s^1\}$ as
\begin{eqnarray}
\left\{
\begin{array}{rcl}
\e^2\dt u^{\e}+\e \vw\cdot\nabla_x u^{\e}+u^{\e}-\bar
u^{\e}&=&0\label{transport}\ \ \ \text{in}\ \
[0,\infty)\times\Omega,\\\rule{0ex}{1.0em}
u^{\e}(0,\vx,\vw)&=&h(\vx,\vw)\ \ \text{in}\ \ \Omega\\\rule{0ex}{1.0em}
u^{\e}(t,\vx_0,\vw)&=&g(t,\vx_0,\vw)\ \ \text{for}\ \ \vw\cdot\vec
n<0\ \ \text{and}\ \ \vx_0\in\p\Omega,
\end{array}
\right.
\end{eqnarray}
where
\begin{eqnarray}\label{average 1}
\bar
u^{\e}(t,\vx)=\frac{1}{2\pi}\int_{\s^1}u^{\e}(t,\vx,\vw)\ud{\vw}.
\end{eqnarray}
and $\vec n$ is the outward normal vector on $\p\Omega$, with the
Knudsen number $0<\e<<1$. The initial and boundary data satisfy the
compatibility condition
\begin{eqnarray}\label{compatibility condition}
h(\vx_0,\vw)=g(0,\vx_0,\vw)\ \ \text{for}\ \ \vw\cdot\vec
n<0\ \ \text{and}\ \ \vx_0\in\p\Omega.
\end{eqnarray}
We intend to study the diffusive limit of $u^{\e}$ as $\e\rt0$.

Based on the flow direction, we can divide the boundary
$\Gamma=\{(\vx,\vw):\ \vx\in\p\Omega\}$ into the in-flow boundary
$\Gamma^-$, the out-flow boundary $\Gamma^+$, and the grazing set
$\Gamma^0$ as
\begin{eqnarray}
\Gamma^{-}&=&\{(\vx,\vw):\ \vx\in\p\Omega,\ \vw\cdot\vec n<0\},\\
\Gamma^{+}&=&\{(\vx,\vw):\ \vx\in\p\Omega,\ \vw\cdot\vec n>0\},\\
\Gamma^{0}&=&\{(\vx,\vw):\ \vx\in\p\Omega,\ \vw\cdot\vec n=0\}.
\end{eqnarray}
It is easy to see $\Gamma=\Gamma^+\cup\Gamma^-\cup\Gamma^0$.

The study of neutron transport equation dates back to 1950s. The main methods include the explicit formula and spectral analysis of the transport operators(see \cite{Larsen1974}, \cite{Larsen1974=}, \cite{Larsen1975}, \cite{Larsen1977}, \cite{Larsen.D'Arruda1976}, \cite{Larsen.Habetler1973}, \cite{Larsen.Keller1974}, \cite{Larsen.Zweifel1974}, \cite{Larsen.Zweifel1976}).  In the classical paper \cite{Bensoussan.Lions.Papanicolaou1979}, a systematic construction of boundary layer was provided via Milne problem. However, this construction was proved to be problematic for steady equation in \cite{AA003} and a new boundary layer construction based on $\e$-Milne problem with geometric correction was presented. In this paper, we extend this result to unsteady equation and consider a more complicated case with initial layer involved.

\subsection{Main Results}

We first present the well-posedness of the equation (\ref{transport}).
\begin{theorem}\label{main 1}
Assume $g(t,x_0,\vw)\in L^{\infty}([0,\infty)\times\Gamma^-)$ and
$h(\vx,\vw)\in L^{\infty}(\Omega\times\s^1)$. Then for the unsteady
neutron transport equation (\ref{transport}), there exists a unique
solution $u^{\e}(t,\vx,\vw)\in
L^{\infty}([0,\infty)\times\Omega\times\s^1)$ satisfying
\begin{eqnarray}\label{main theorem 1}
\im{u^{\e}}{[0,\infty)\times\Omega\times\s^1}\leq
C(\Omega)\bigg(\nm{h}_{L^{\infty}(\Omega\times\s^1)}+\nm{g}_{L^{\infty}([0,\infty)\times\Gamma^-)}\bigg).
\end{eqnarray}
\end{theorem}
Then we can show the diffusive limit of the equation (\ref{transport}).
\begin{theorem}\label{main 2}
Assume $g(t,\vx_0,\vw)\in C^4([0,\infty)\times\Gamma^-)$ and
$h(\vx,\vw)\in C^4(\Omega\times\s^1)$. Then for the unsteady neutron transport
equation (\ref{transport}), the unique solution
$u^{\e}(t,\vx,\vw)\in L^{\infty}([0,\infty)\times\Omega\times\s^1)$
satisfies
\begin{eqnarray}\label{main theorem 2}
\lnm{u^{\e}-\u_0-\ub_{I,0}-\ub_{B,0}}=o(1),
\end{eqnarray}
where the interior solution $\u_0$ is
defined in (\ref{expansion temp 8}), the initial layer $\ub_{I,0}$ is defined in (\ref{expansion temp 21}), and the boundary layer $\ub_{B,0}$ is defined in (\ref{expansion temp 9}). Moreover, if $g(t,\theta,\phi)=t^2\ue^{-t}\cos\phi$ and $h(\vx,\vw)=0$, then there exists
a $C>0$ such that
\begin{eqnarray}\label{main theorem 3}
\lnm{u^{\e}-\uc_0-\ubc_{I,0}-\ubc_{B,0}}\geq C>0
\end{eqnarray}
when $\e$ is sufficiently small, where the interior solution $\uc_0$ is
defined in (\ref{classical temp 2.}), the initial layer $\ubc_{I,0}$ is defined in (\ref{classical temp 21.}), and the boundary layer $\ub_{B,0}$ is defined in (\ref{classical temp 1.}).
\end{theorem}
\begin{remark}
$\theta$ and $\phi$ are defined in (\ref{substitution 2}) and
(\ref{substitution 4}).
\end{remark}
It is easy to see, by a similar argument, the results in Theorem \ref{main 1} and Theorem \ref{main 2} also hold for the one-dimensional unsteady neutron transport equation, where the temporal domain is $[0,\infty)$, spacial domain is $[0,L]$ for fixed $L>0$, and velocity domain is $[-1/2,1/2]$.

\subsection{Notation and Structure of This Paper}

Throughout this paper, $C>0$ denotes a constant that only depends on
the parameter $\Omega$, but does not depend on the data. It is
referred as universal and can change from one inequality to another.
When we write $C(z)$, it means a certain positive constant depending
on the quantity $z$. We write $a\ls b$ to denote $a\leq Cb$.

Our paper is organized as follows: in Section 2, we establish the $L^{\infty}$ well-posedness of
the equation (\ref{transport}) and prove Theorem \ref{main 1}; in Section 3, we present the asymptotic analysis of the equation (\ref{transport});  in Section 4, we give the main
results of the $\e$-Milne problem with geometric correction; in Section 5, we prove the first part of Theorem \ref{main 2}; finally, in
Section 6, we prove the second part of Theorem \ref{main 2}.
%
%

\section{Well-posedness of Unsteady Neutron Transport Equation}

In this section, we consider the well-posedness of the unsteady
neutron transport equation
\begin{eqnarray}
\left\{ \begin{array}{rcl} \e^2\dt u+\e\vec w\cdot\nabla_xu+u-\bar
u&=&f(t,\vx,\vw)\ \ \text{in}\ \
[0,\infty)\times\Omega\label{neutron},\\\rule{0ex}{1.0em}
u(0,\vx,\vw)&=&h(\vx,\vw)\ \ \text{in}\ \ \Omega\\\rule{0ex}{1.0em}
u(t,\vec x_0,\vec w)&=&g(t,\vec x_0,\vec w)\ \ \text{for}\ \ \vw\cdot\vec n<0\ \
\text{and}\ \ \vx_0\in\p\Omega,
\end{array}
\right.
\end{eqnarray}
where the initial and boundary data satisfy the compatibility
condition
\begin{eqnarray}
h(\vx_0,\vw)=g(0,\vx_0,\vw)\ \ \text{for}\ \ \vw\cdot\vec n<0\ \
\text{and}\ \ \vx_0\in\p\Omega.
\end{eqnarray}
We define the $L^2$ and $L^{\infty}$ norms in $\Omega\times\s^1$ as
usual:
\begin{eqnarray}
\nm{f}_{L^2(\Omega\times\s^1)}&=&\bigg(\int_{\Omega}\int_{\s^1}\abs{f(\vx,\vw)}^2\ud{\vw}\ud{\vx}\bigg)^{1/2},\\
\nm{f}_{L^{\infty}(\Omega\times\s^1)}&=&\sup_{(\vx,\vw)\in\Omega\times\s^1}\abs{f(\vx,\vw)}.
\end{eqnarray}
Define the $L^2$ and $L^{\infty}$ norms on the boundary as follows:
\begin{eqnarray}
\nm{f}_{L^2(\Gamma)}&=&\bigg(\iint_{\Gamma}\abs{f(\vx,\vw)}^2\abs{\vw\cdot\vec n}\ud{\vw}\ud{\vx}\bigg)^{1/2},\\
\nm{f}_{L^2(\Gamma^{\pm})}&=&\bigg(\iint_{\Gamma^{\pm}}\abs{f(\vx,\vw)}^2\abs{\vw\cdot\vec
n}\ud{\vw}\ud{\vx}\bigg)^{1/2},\\
\nm{f}_{L^{\infty}(\Gamma)}&=&\sup_{(\vx,\vw)\in\Gamma}\abs{f(\vx,\vw)},\\
\nm{f}_{L^{\infty}(\Gamma^{\pm})}&=&\sup_{(\vx,\vw)\in\Gamma^{\pm}}\abs{f(\vx,\vw)}.
\end{eqnarray}
Similar notation also applies to the space
$[0,\infty)\times\Omega\times\s^1$, $[0,\infty)\times\Gamma$, and
$[0,\infty)\times\Gamma^{\pm}$.

\subsection{Preliminaries}

In order to show the $L^2$ and $L^{\infty}$ well-posedness of the
equation (\ref{neutron}), we start with some preparations of the
penalized neutron transport equation.
\begin{lemma}\label{well-posedness lemma 1}
Assume $f(t,\vx,\vw)\in
L^{\infty}([0,\infty)\times\Omega\times\s^1)$, $h(\vx,\vw)\in
L^{\infty}(\Omega\times\s^1)$ and $g(t,x_0,\vw)\in
L^{\infty}([0,\infty)\times\Gamma^-)$. Then for the penalized
transport equation
\begin{eqnarray}\label{penalty equation}
\left\{
\begin{array}{rcl}
\lambda u_{\l}+\e^2\dt
u_{\l}+\e\vw\cdot\nabla_xu_{\l}+u_{\l}&=&f(t,\vx,\vw)\ \ \text{in}\
\ [0,\infty)\times\Omega,\\\rule{0ex}{1.0em}
u_{\l}(0,\vx,\vw)&=&h(\vx,\vw)\ \ \text{in}\ \
\Omega\\\rule{0ex}{1.0em} u_{\l}(t,\vx_0,\vec w)&=&g(t,\vx_0,\vw)\ \ \text{for}\ \ \vw\cdot\vec n<0\ \
\text{and}\ \ \vx_0\in\p\Omega.
\end{array}
\right.
\end{eqnarray}
with $\l>0$ as a penalty parameter, there exists a solution
$u_{\l}(t,\vx,\vw)\in L^{\infty}([0,T]\times\Omega\times\s^1)$
satisfying
\begin{eqnarray}
\im{u_{\l}}{[0,\infty)\times\Omega\times\s^1}\leq
\im{g}{[0,\infty)\times\Gamma^-}+\im{h}{\Omega\times\s^1}+\im{f}{[0,\infty)\times\Omega\times\s^1}.
\end{eqnarray}
\end{lemma}
\begin{proof}
The characteristics $(T(s),X(s),W(s))$ of the equation (\ref{penalty
equation}) which goes through $(t,\vx,\vw)$ is defined by
\begin{eqnarray}\label{character}
\left\{
\begin{array}{rcl}
(T(0),X(0),W(0))&=&(t,\vx,\vw)\\\rule{0ex}{2.0em}
\dfrac{\ud{T(s)}}{\ud{s}}&=&\e^2,\\\rule{0ex}{2.0em}
\dfrac{\ud{X(s)}}{\ud{s}}&=&\e W(s),\\\rule{0ex}{2.0em}
\dfrac{\ud{W(s)}}{\ud{s}}&=&0.
\end{array}
\right.
\end{eqnarray}
which implies
\begin{eqnarray}
\left\{
\begin{array}{rcl}
T(s)&=&t+\e^2s,\\
X(s)&=&\vx+(\e\vw)s,\\
W(s)&=&\vw,
\end{array}
\right.
\end{eqnarray}
Hence, we can rewrite the equation (\ref{penalty equation}) along
the characteristics as
\begin{eqnarray}\label{well-posedness temp 31}
&&u_{\l}(t,\vx,\vw)\\
&=&{\bf 1}_{\{t\geq \e^2t_b\}}\bigg( g(t-\e^2t_b,\vx-\e
t_b\vw,\vw)\ue^{-(1+\l)t_b}+\int_{0}^{t_b}f(t-\e^2(t_b-s),\vx-\e(t_b-s)\vw,\vw)\ue^{-(1+\l)(t_b-s)}\ud{s}\bigg)\no\\
&&+{\bf 1}_{\{t\leq \e^2t_b\}}\bigg( h(\vx-(\e
t\vw)/\e^2,\vw)\ue^{-(1+\l)t/\e^2}+\int_{0}^{t/\e^2}f(\e^2s,\vx-\e(t/\e^2-s)\vw,\vw)\ue^{-(1+\l)(t/\e^2-s)}\ud{s}\bigg)\no,
\end{eqnarray}
where the backward exit time $t_b$ is defined as
\begin{equation}\label{exit time}
t_b(\vx,\vw)=\inf\{s\geq0: (\vx-\e s\vw,\vw)\in\Gamma^-\}.
\end{equation}
Then we can naturally estimate
\begin{eqnarray}
&&\im{u_{\l}}{[0,\infty)\times\Omega\times\s^1}\\
&\leq&{\bf 1}_{\{t\geq \e^2t_b\}}\bigg(\ue^{-(1+\l)t_b}\im{g}{[0,\infty)\times\Gamma^-}+\frac{1-\ue^{(1+\l)t_b}}{1+\l}\im{f}{[0,\infty)\times\Omega\times\s^1}\bigg)\no\\
&&+{\bf 1}_{\{t\leq \e^2t_b\}}\bigg(\ue^{-(1+\l)t/\e^2}\im{h}{\Omega\times\s^1}+\frac{1-\ue^{(1+\l)t/\e^2}}{1+\l}\im{f}{[0,\infty)\times\Omega\times\s^1}\bigg)\no\\
&\leq&\im{g}{[0,\infty)\times\Gamma^-}+\im{h}{\Omega\times\s^1}+\im{f}{[0,\infty)\times\Omega\times\s^1}\nonumber.
\end{eqnarray}
Since $u_{\l}$ can be explicitly traced back to the initial or
boundary data, the existence naturally follows from above estimate.
\end{proof}
\begin{lemma}\label{well-posedness lemma 2}
Assume $f(t,\vx,\vw)\in
L^{\infty}([0,\infty)\times\Omega\times\s^1)$, $h(\vx,\vw)\in
L^{\infty}(\Omega\times\s^1)$ and $g(t,x_0,\vw)\in
L^{\infty}([0,\infty)\times\Gamma^-)$. Then for the penalized
neutron transport equation
\begin{eqnarray}
\left\{ \begin{array}{rcl} \l u_{\l}+\e^2\dt
u_{\l}+\e\vw\cdot\nabla_xu_{\l}+u_{\l}-\bar u_{\l}&=& f(t,\vx,\vw)\
\ \text{in}\ \ [0,\infty)\times\Omega\label{well-posedness penalty
equation},\\\rule{0ex}{1.0em} u_{\l}(0,\vx,\vw)&=&h(\vx,\vw)\ \
\text{in}\ \ \Omega\\\rule{0ex}{1.0em} u_{\l}(t,\vx_0,\vec
w)&=&g(t,\vx_0,\vw)\ \ \text{for}\ \ \vx_0\in\p\Omega\ \ \text{and}\
\vw\cdot\vec n<0.
\end{array}
\right.
\end{eqnarray}
with $\l>0$, there exists a solution $u_{\l}(t,\vx,\vw)\in
L^{\infty}([0,\infty)\times\Omega\times\s^1)$ satisfying
\begin{eqnarray}
\im{u_{\l}}{[0,\infty)\times\Omega\times\s^1}\leq
\frac{1+\e}{\l}\bigg(\im{g}{[0,\infty)\times\Gamma^-}+\im{h}{\Omega\times\s^1}+\im{f}{[0,\infty)\times\Omega\times\s^1}\bigg).
\end{eqnarray}
\end{lemma}
\begin{proof}
We define an approximating sequence $\{u_{\l}^k\}_{k=0}^{\infty}$,
where $u_{\l}^0=0$ and
\begin{eqnarray}\label{penalty temp 1}
\left\{
\begin{array}{rcl}
\l u_{\l}^{k}+\e^2\dt
u_{\l}^k+\e\vw\cdot\nabla_xu_{\l}^k+u_{\l}^k-\bar
u_{\l}^{k-1}&=&f(t,\vx,\vw)\ \ \text{in}\ \
[0,\infty)\times\Omega,\\\rule{0ex}{1.0em}
u_{\l}^k(0,\vx,\vw)&=&h(\vx,\vw)\ \ \text{in}\ \
\Omega\\\rule{0ex}{1.0em} u_{\l}^k(t,\vx_0,\vw)&=&g(t,\vx_0,\vw)\ \
\text{for}\ \ \vx_0\in\p\Omega\ \ \text{and}\ \vw\cdot\vec n<0.
\end{array}
\right.
\end{eqnarray}
By Lemma \ref{well-posedness lemma 1}, this sequence is well-defined
and $\im{u_{\l}^k}{[0,\infty)\times\Omega\times\s^1}<\infty$.

The characteristics and the backward exit time are defined as
(\ref{character}) and (\ref{exit time}), so we rewrite equation
(\ref{penalty temp 1}) along the characteristics as
\begin{eqnarray}
\
\end{eqnarray}
\begin{eqnarray}
&&u_{\l}^k(t,\vx,\vw)\no\\
&=&{\bf 1}_{\{t\geq \e^2t_b\}}\bigg( g(t-\e^2t_b,\vx-\e
t_b\vw,\vw)\ue^{-(1+\l)t_b}+\int_{0}^{t_b}(\bar u_{\l}^{k-1}+f)(t-\e^2(t_b-s),\vx-\e(t_b-s)\vw,\vw)\ue^{-(1+\l)(t_b-s)}\ud{s}\bigg)\no\\
&&+{\bf 1}_{\{t\leq \e^2t_b\}}\bigg( h(\vx-(\e
t\vw)/\e^2,\vw)\ue^{-(1+\l)t/\e^2}+\int_{0}^{t/\e^2}(\bar
u_{\l}^{k-1}+f)(\e^2s,\vx-\e(t/\e^2-s)\vw,\vw)\ue^{-(1+\l)(t/\e^2-s)}\ud{s}\bigg)\no,
\end{eqnarray}
We define the difference $v^k=u_{\l}^{k}-u_{\l}^{k-1}$ for $k\geq1$.
Then $v^k$ satisfies
\begin{eqnarray}
v^{k+1}(\vx,\vw)&=&{\bf 1}_{\{t\geq \e^2t_b\}}\bigg( \int_{0}^{t_b}\bar v_{\l}^{k-1}(t-\e^2(t_b-s),\vx-\e(t_b-s)\vw,\vw)\ue^{-(1+\l)(t_b-s)}\ud{s}\bigg)\\
&&+{\bf 1}_{\{t\leq \e^2t_b\}}\bigg(\int_{0}^{t/\e^2}\bar
v_{\l}^{k-1}(\e^2s,\vx-\e(t/\e^2-s)\vw,\vw)\ue^{-(1+\l)(t/\e^2-s)}\ud{s}\bigg)\no,
\end{eqnarray}
Since $\im{\bar
v^k}{[0,\infty)\times\Omega\times\s^1}\leq\im{v^k}{[0,\infty)\times\Omega\times\s^1}$,
we can directly estimate
\begin{eqnarray}
\im{v^{k+1}}{[0,\infty)\times\Omega\times\s^1}&\leq&\im{v^{k}}{[0,\infty)\times\Omega\times\s^1}\int_0^{\max\{t/\e^2,t_b\}}\ue^{-(1+\l)(t_b-s)}\ud{s}\\
&\leq&\frac{1-\ue^{-(1+\l)\max\{t/\e^2,t_b\}}}{1+\l}\im{v^{k}}{[0,\infty)\times\Omega\times\s^1}.\no
\end{eqnarray}
Hence, we naturally have
\begin{eqnarray}
\im{v^{k+1}}{[0,\infty)\times\Omega\times\s^1}&\leq&\frac{1}{1+\l}\im{v^{k}}{[0,\infty)\times\Omega\times\s^1}.
\end{eqnarray}
Thus, this is a contraction sequence for $\l>0$. Considering
$v^1=u_{\l}^1$, we have
\begin{eqnarray}
\im{v^{k}}{[0,\infty)\times\Omega\times\s^1}\leq\bigg(\frac{1}{1+\l}\bigg)^{k-1}\im{u^{1}_{\l}}{[0,\infty)\times\Omega\times\s^1},
\end{eqnarray}
for $k\geq1$. Therefore, $u_{\l}^k$ converges strongly in
$L^{\infty}$ to a limit solution $u_{\l}$ satisfying
\begin{eqnarray}\label{well-posedness temp 1}
\im{u_{\l}}{[0,\infty)\times\Omega\times\s^1}\leq\sum_{k=1}^{\infty}\im{v^{k}}{[0,\infty)\times\Omega\times\s^1}\leq\frac{1+\l}{\l}\im{u_{\l}^1}{[0,\infty)\times\Omega\times\s^1}.
\end{eqnarray}
Since $u_{\l}^1$ can be rewritten along the characteristics as
\begin{eqnarray}
&&u_{\l}^1(\vx,\vw)\\
&=&{\bf 1}_{\{t\geq \e^2t_b\}}\bigg( g(t-\e^2t_b,\vx-\e
t_b\vw,\vw)\ue^{-(1+\l)t_b}+\int_{0}^{t_b}f(t-\e^2(t_b-s),\vx-\e(t_b-s)\vw,\vw)\ue^{-(1+\l)(t_b-s)}\ud{s}\bigg)\no\\
&&+{\bf 1}_{\{t\leq \e^2t_b\}}\bigg( h(\vx-(\e
t\vw)/\e^2,\vw)\ue^{-(1+\l)t/\e^2}+\int_{0}^{t/\e^2}f(\e^2s,\vx-\e(t/\e^2-s)\vw,\vw)\ue^{-(1+\l)(t/\e^2-s)}\ud{s}\bigg)\no,
\end{eqnarray}
based on Lemma \ref{well-posedness lemma 1}, we can directly
estimate
\begin{eqnarray}\label{well-posedness temp 2}
\im{u_{\l}^1}{[0,\infty)\times\Omega\times\s^1}\leq
\im{g}{[0,\infty)\times\Gamma^-}+\im{h}{\Omega\times\s^1}+\im{f}{[0,\infty)\times\Omega\times\s^1}.
\end{eqnarray}
Combining (\ref{well-posedness temp 1}) and (\ref{well-posedness
temp 2}), we can easily deduce the lemma.
\end{proof}

\subsection{$L^2$ Estimate}

It is easy to see when $\l\rt0$, the estimate in Lemma
\ref{well-posedness lemma 2} blows up. Hence, we need to show a
uniform estimate of the solution to the penalized neutron transport
equation (\ref{well-posedness penalty equation}).
\begin{lemma}(Green's Identity)\label{well-posedness lemma 3}
Assume $f(t,\vx,\vw),\ g(t,\vx,\vw)\in
L^{\infty}([0,\infty)\times\Omega\times\s^1)$ and $\dt f+\vw\cdot\nx
f,\ \dt g+\vw\cdot\nx g\in L^2([0,\infty)\times\Omega\times\s^1)$
with $f,\ g\in L^2([0,\infty)\times\Gamma)$. Then for almost all
$s,t\in[0,\infty)$,
\begin{eqnarray}
&&\int_s^t\iint_{\Omega\times\s^1}\bigg((\dt f+\vw\cdot\nx f)g+(\dt
g+\vw\cdot\nx
g)f\bigg)\ud{\vx}\ud{\vw}\ud{r}\\
&=&\int_s^t\int_{\Gamma}fg\ud{\gamma}\ud{r}+\iint_{\Omega\times\s^1}f(t)g(t)\ud{\vx}\ud{\vw}-\iint_{\Omega\times\s^1}f(s)g(s)\ud{\vx}\ud{\vw},\no
\end{eqnarray}
where $\ud{\gamma}=(\vw\cdot\vec n)\ud{s}$ on the boundary.
\end{lemma}
\begin{proof}
See \cite[Chapter 9]{Cercignani.Illner.Pulvirenti1994} and
\cite{Esposito.Guo.Kim.Marra2013}.
\end{proof}
\begin{lemma}\label{well-posedness lemma 4}
The solution $u_{\l}$ to the equation (\ref{well-posedness penalty
equation}) satisfies the uniform estimate in time interval $[s,t]$,
\begin{eqnarray}\label{well-posedness temp 3}
\\
\e\tm{\bar u_{\l}}{[s,t]\times\Omega\times\s^1} &\leq&
C(\Omega)\bigg( \tm{u_{\l}-\bar
u_{\l}}{[s,t]\times\Omega\times\s^1}+\tm{f}{[s,t]\times\Omega\times\s^1}+\e\tm{u_{\l}}{[s,t]\times\Gamma^{+}}\no\\
&&+\e\tm{g}{[s,t]\times\Gamma^-}\bigg)+\e^2G(t)-\e^2G(s),\no
\end{eqnarray}
where $G(t)$ is a function satisfying
\begin{eqnarray}
G(t)\leq C(\Omega)\tm{u_{\l}(t)}{\Omega\times\s^1},
\end{eqnarray}
for $0\leq\l<<1$ and $0<\e<<1$.
\end{lemma}
\begin{proof}
We divide the proof into several steps:\\
\ \\
Step 1:\\
Applying Lemma \ref{well-posedness lemma 3} to the solution of the
equation (\ref{well-posedness penalty equation}). Then for any
$\phi\in L^{\infty}([0,\infty)\times\Omega\times\s^1)$ satisfying
$\e\dt\phi+\vw\cdot\nx\phi\in L^2([0,\infty)\times\Omega\times\s^1)$
and $\phi\in L^{2}([0,\infty)\times\Gamma)$, we have
\begin{eqnarray}\label{well-posedness temp 4}
&&\l\int_s^t\iint_{\Omega\times\s^1}u_{\l}\phi
-\e^2\int_s^t\iint_{\Omega\times\s^1}\dt\phi
u_{\l}-\e\int_s^t\iint_{\Omega\times\s^1}(\vw\cdot\nx\phi)u_{\l}+\int_s^t\iint_{\Omega\times\s^1}(u_{\l}-\bar
u_{\l})\phi\\
&=&-\e\int_s^t\int_{\Gamma}u_{\l}\phi\ud{\gamma}-\e^2\iint_{\Omega\times\s^1}u_{\l}(t)\phi(t)+\e^2\iint_{\Omega\times\s^1}u_{\l}(s)\phi(s)+\int_s^t\iint_{\Omega\times\s^1}f\phi.\no
\end{eqnarray}
Our goal is to choose a particular test function $\phi$. We first
construct an auxiliary function $\zeta(t)$. Since $u_{\l}(t)\in
L^{\infty}(\Omega\times\s^1)$, it naturally implies $\bar
u_{\l}(t)\in L^{\infty}(\Omega)$ which further leads to $\bar
u_{\l}(t)\in L^2(\Omega)$. We define $\zeta(t,\vx)$ on $\Omega$
satisfying
\begin{eqnarray}\label{test temp 1}
\left\{
\begin{array}{rcl}
\Delta \zeta(t)&=&\bar u_{\l}(t)\ \ \text{in}\ \
\Omega,\\\rule{0ex}{1.0em} \zeta(t)&=&0\ \ \text{on}\ \ \p\Omega.
\end{array}
\right.
\end{eqnarray}
In the bounded domain $\Omega$, based on the standard elliptic
estimate, we have
\begin{eqnarray}\label{test temp 3}
\nm{\zeta(t)}_{H^2(\Omega)}\leq C(\Omega)\nm{\bar
u_{\l}(t)}_{L^2(\Omega)}.
\end{eqnarray}
\ \\
Step 2:\\
Without loss of generality, we only prove the case with $s=0$. We
plug the test function
\begin{eqnarray}\label{test temp 2}
\phi(t)=-\vw\cdot\nx\zeta(t)
\end{eqnarray}
into the weak formulation (\ref{well-posedness temp 4}) and estimate
each term there. Naturally, we have
\begin{eqnarray}\label{test temp 4}
\nm{\phi(t)}_{L^2(\Omega)}\leq C\nm{\zeta(t)}_{H^1(\Omega)}\leq
C(\Omega)\nm{\bar u_{\l}(t)}_{L^2(\Omega)}.
\end{eqnarray}
Easily we can decompose
\begin{eqnarray}\label{test temp 5}
-\e\int_0^t\iint_{\Omega\times\s^1}(\vw\cdot\nx\phi)u_{\l}&=&-\e\int_0^t\iint_{\Omega\times\s^1}(\vw\cdot\nx\phi)\bar
u_{\l}-\e\int_0^t\iint_{\Omega\times\s^1}(\vw\cdot\nx\phi)(u_{\l}-\bar
u_{\l}).
\end{eqnarray}
We estimate the two term on the right-hand side of (\ref{test temp
5}) separately. By (\ref{test temp 1}) and (\ref{test temp 2}), we
have
\begin{eqnarray}\label{wellposed temp 1}
\\
-\e\int_0^t\iint_{\Omega\times\s^1}(\vw\cdot\nx\phi)\bar
u_{\l}&=&\e\int_0^t\iint_{\Omega\times\s^1}\bar
u_{\l}\bigg(w_1(w_1\p_{11}\zeta+w_2\p_{12}\zeta)+w_2(w_1\p_{12}\zeta+w_2\p_{22}\zeta)\bigg)\no\\
&=&\e\int_0^t\iint_{\Omega\times\s^1}\bar
u_{\l}\bigg(w_1^2\p_{11}\zeta+w_2^2\p_{22}\zeta\bigg)\nonumber\\
&=&\e\pi\int_0^t\int_{\Omega}\bar u_{\l}(\p_{11}\zeta+\p_{22}\zeta)\nonumber\\
&=&\e\pi\nm{\bar u_{\l}}_{L^2([0,t]\times\Omega)}^2\nonumber\\
&=&\half\e\nm{\bar
u_{\l}}_{L^2([0,t]\times\Omega\times\s^1)}^2\nonumber.
\end{eqnarray}
In the second equality, above cross terms vanish due to the symmetry
of the integral over $\s^1$. On the other hand, for the second term
in (\ref{test temp 5}), H\"older's inequality and the elliptic
estimate imply
\begin{eqnarray}\label{wellposed temp 2}
-\e\int_0^t\iint_{\Omega\times\s^1}(\vw\cdot\nx\phi)(u_{\l}-\bar
u_{\l})&\leq&C(\Omega)\e\nm{u_{\l}-\bar u_{\l}}_{L^2([0,t]\times\Omega\times\s^1)}\bigg(\int_0^t\nm{\zeta}^2_{H^2(\Omega)}\bigg)^{1/2}\\
&\leq&C(\Omega)\e\nm{u_{\l}-\bar
u_{\l}}_{L^2([0,t]\times\Omega\times\s^1)}\nm{\bar
u_{\l}}_{L^2([0,t]\times\Omega\times\s^1)}\nonumber.
\end{eqnarray}
Based on (\ref{test temp 3}), (\ref{test temp 4}), the boundary
condition of the penalized neutron transport equation
(\ref{well-posedness penalty equation}), the trace theorem,
H\"older's inequality and the elliptic estimate, we have
\begin{eqnarray}\label{wellposed temp 3}
\\
\e\int_0^t\int_{\Gamma}u_{\l}\phi\ud{\gamma}&=&\e\int_0^t\int_{\Gamma^+}u_{\l}\phi\ud{\gamma}+\e\int_0^t\int_{\Gamma^-}u_{\l}\phi\ud{\gamma}\no\\
&\leq&C(\Omega)\bigg(\e\nm{u_{\l}}_{L^2([0,t]\times\Gamma^+)}\nm{\bar
u_{\l}}_{L^2([0,t]\times\Omega\times\s^1)}+\e\tm{g}{[0,t]\times\Gamma^-}\nm{\bar
u_{\l}}_{L^2([0,t]\times\Omega\times\s^1)}\bigg)\nonumber,
\end{eqnarray}
\begin{eqnarray}\label{wellposed temp 4}
\l\int_0^t\iint_{\Omega\times\s^1}u_{\l}\phi&=&\l\int_0^t\iint_{\Omega\times\s^1}\bar
u_{\l}\phi+\l\int_0^t\iint_{\Omega\times\s^1}(u_{\l}-\bar u_{\l})\phi=\l\int_0^t\iint_{\Omega\times\s^1}(u_{\l}-\bar u_{\l})\phi\\
&\leq&C(\Omega)\l\nm{\bar
u_{\l}}_{L^2([0,t]\times\Omega\times\s^1)}\nm{u_{\l}-\bar
u_{\l}}_{L^2([0,t]\times\Omega\times\s^1)}\nonumber,
\end{eqnarray}
\begin{eqnarray}\label{wellposed temp 5}
\int_0^t\iint_{\Omega\times\s^1}(u_{\l}-\bar u_{\l})\phi\leq
C(\Omega)\nm{\bar
u_{\l}}_{L^2([0,t]\times\Omega\times\s^1)}\nm{u_{\l}-\bar
u_{\l}}_{L^2([0,t]\times\Omega\times\s^1)},
\end{eqnarray}
\begin{eqnarray}\label{wellposed temp 6}
\int_0^t\iint_{\Omega\times\s^1}f\phi\leq C(\Omega)\nm{\bar
u_{\l}}_{L^2([0,t]\times\Omega\times\s^1)}\nm{f}_{L^2([0,t]\times\Omega\times\s^1)}.
\end{eqnarray}
Note that we will take
\begin{eqnarray}\label{wellposed temp 7}
-\e^2\iint_{\Omega\times\s^1}u_{\l}(t)\phi(t)+\e^2\iint_{\Omega\times\s^1}u_{\l}(0)\phi(0)=\e^2\bigg(G(t)-G(0)\bigg),
\end{eqnarray}
where $G(t)=-\displaystyle\iint_{\Omega\times\s^1}u_{\l}(t)\phi(t)$.
Then the only remaining term is
\begin{eqnarray}\label{wellposed temp 8}
-\e^2\int_0^t\iint_{\Omega\times\s^1}\dt\phi u_{\l}&=&-\e^2\int_s^t\iint_{\Omega\times\s^1}\dt\phi (u_{\l}-\bar u_{\l})\\
&\leq&\nm{\dt\nabla\zeta}_{L^2([0,t]\times\Omega\times\s^1)}\nm{u_{\l}-\bar
u_{\l}}_{L^2([0,t]\times\Omega\times\s^1)}.\no
\end{eqnarray}
Now we have to tackle $\nm{\dt\nabla\zeta}_{L^2([0,t]\times\Omega\times\s^1)}$.\\
\ \\
Step 3:\\
For test function $\phi(\vx,\vw)$ which is independent of time $t$,
in time interval $[t-\delta,t]$ the weak formulation in
(\ref{well-posedness temp 4}) can be simplified as
\begin{eqnarray}\label{test temp 6}
&&\l\int_{t-\delta}^t\iint_{\Omega\times\s^1}u_{\l}\phi
-\e\int_{t-\delta}^t\iint_{\Omega\times\s^1}(\vw\cdot\nx\phi)u_{\l}+\int_{t-\delta}^t\iint_{\Omega\times\s^1}(u_{\l}-\bar
u_{\l})\phi\\
&=&-\e\int_{t-\delta}^t\int_{\Gamma}u_{\l}\phi\ud{\gamma}
-\e^2\iint_{\Omega\times\s^1}u_{\l}(t)\phi+\e^2\iint_{\Omega\times\s^1}u_{\l}({t-\delta})\phi+\int_{t-\delta}^t\iint_{\Omega\times\s^1}f\phi.\no
\end{eqnarray}
Taking difference quotient as $\delta\rt0$, we know
\begin{eqnarray}
\frac{\e^2\displaystyle\iint_{\Omega\times\s^1}u_{\l}(t)\phi-\e^2\displaystyle\iint_{\Omega\times\s^1}u_{\l}({t-\delta})\phi}{\delta}\rt
\e^2\iint_{\Omega\times\s^1}\dt
u_{\l}(t)\phi.
\end{eqnarray}
Then (\ref{test temp 6}) can be simplified into
\begin{eqnarray}\label{test temp 7}
&&\e^2\iint_{\Omega\times\s^1}\dt u_{\l}(t)\phi\\
&=&-\l\iint_{\Omega\times\s^1}u_{\l}(t)\phi
+\e\iint_{\Omega\times\s^1}(\vw\cdot\nx\phi)u_{\l}(t)-\iint_{\Omega\times\s^1}(u_{\l}(t)-\bar
u_{\l}(t))\phi\no\\
&&-\e\int_{\Gamma}u_{\l}(t)\phi\ud{\gamma}+\iint_{\Omega\times\s^1}f(t)\phi.\no
\end{eqnarray}
For fixed $t$, taking $\phi=\Phi(\vx)$ which satisfies
\begin{eqnarray}
\left\{
\begin{array}{rcl}
\Delta \Phi&=&\dt\bar u_{\l}(t)\ \ \text{in}\ \
\Omega,\\\rule{0ex}{1.0em} \Phi&=&0\ \ \text{on}\ \ \p\Omega,
\end{array}
\right.
\end{eqnarray}
which further implies $\Phi=\dt\zeta$.
Then the left-hand side of (\ref{test temp 7}) is actually
\begin{eqnarray}
LHS&=&\e^2\iint_{\Omega\times\s^1}\Phi\dt u_{\l}(t)=\e^2\iint_{\Omega\times\s^1}\Phi\dt\bar u_{\l}\\
&=&\e^2\iint_{\Omega\times\s^1}\Phi\Delta\Phi=\e^2\iint_{\Omega\times\s^1}\abs{\nabla\Phi}^2\no\\
&=&\nm{\dt\nabla\zeta(t)}_{L^2(\Omega\times\s^1)}^2.\no
\end{eqnarray}
By a similar argument as in Step 2 and the Poincar\'e inequality, the right-hand side of
(\ref{test temp 7}) can be bounded as
\begin{eqnarray}
\\
RHS\ls \nm{\dt\nabla\zeta(t)}_{L^2(\Omega\times\s^1)}\bigg(\nm{u_{\l}(t)-\bar
u_{\l}(t)}_{L^2(\Omega\times\s^1)}+\l\nm{\bar
u_{\l}(t)}_{L^2(\Omega\times\s^1)}
+\nm{f(t)}_{L^2(\Omega\times\s^1)}\bigg).\no
\end{eqnarray}
Therefore, we have
\begin{eqnarray}
\nm{\dt\nabla\zeta(t)}_{L^2(\Omega\times\s^1)}\ls
\nm{u_{\l}(t)-\bar u_{\l}(t)}_{L^2(\Omega\times\s^1)}+\l\nm{\bar
u_{\l}(t)}_{L^2(\Omega\times\s^1)}
+\nm{f(t)}_{L^2(\Omega\times\s^1)}.
\end{eqnarray}
For all $t$, we can further integrate over $[0,t]$ to obtain
\begin{eqnarray}\label{wellposed temp 9}
&&\nm{\dt\nabla\zeta(t)}_{L^2([0,t]\times\Omega\times\s^1)}\\
&\ls& \nm{u_{\l}(t)-\bar
u_{\l}(t)}_{L^2([0,t]\times\Omega\times\s^1)}+\l\nm{\bar
u_{\l}(t)}_{L^2([0,t]\times\Omega\times\s^1)}
+\nm{f(t)}_{L^2([0,t]\times\Omega\times\s^1)}.\no
\end{eqnarray}
\ \\
Step 4:\\
Collecting terms in (\ref{wellposed temp 1}), (\ref{wellposed temp
2}), (\ref{wellposed temp 3}), (\ref{wellposed temp 4}),
(\ref{wellposed temp 5}), (\ref{wellposed temp 6}), (\ref{wellposed
temp 7}), (\ref{wellposed temp 8}), and (\ref{wellposed temp 9}), we
obtain
\begin{eqnarray}
&&\e\nm{\bar u_{\l}}_{L^2([0,t]\times\Omega\times\s^1)}\\
&\leq& C(\Omega)\bigg((1+\e+\l)\nm{u_{\l}-\bar
u_{\l}}_{L^2([0,t]\times\Omega\times\s^1)}+\e\nm{u_{\l}}_{L^2([0,t]\times\Gamma^+)}+\nm{f}_{L^2([0,t]\times\Omega\times\s^1)}+\e\tm{g}{[0,t]\times\Gamma^-}\bigg)\nonumber\\
&&+\e^2G(t)-\e^2G(0).\no
\end{eqnarray}
When $0\leq\l<1$ and $0<\e<1$, we get the desired uniform estimate
with respect to $\lambda$.
\end{proof}
\begin{theorem}\label{LT estimate}
Assume $\ue^{\l_0 t}f(t,\vx,\vw)\in
L^{2}([0,\infty)\times\Omega\times\s^1)$, $h(\vx,\vw)\in
L^{2}(\Omega\times\s^1)$ and $\ue^{\l_0 t}g(t,x_0,\vw)\in
L^{2}([0,\infty)\times\Gamma^-)$ for some $\l_0>0$. Then for the unsteady neutron
transport equation (\ref{neutron}), there exists $\l_0^{\ast}$ satisfying $0<\l_0^{\ast}\leq\l_0$ and a unique solution
$u(t,\vx,\vw)\in L^2([0,\infty)\times\Omega\times\s^1)$ satisfying
\begin{eqnarray}
&&\frac{1}{\e^{1/2}}\nm{\ue^{\l t}u}_{L^2([0,\infty)\times\Gamma^+)}+\nm{\ue^{\l t}u(t)}_{L^2(\Omega\times\s^1)}+\nm{\ue^{\l t}u}_{L^2([0,\infty)\times\Omega\times\s^1)}\\
&\leq& C(\Omega)\bigg( \frac{1}{\e^2}\nm{\ue^{\l
t}f}_{L^2([0,\infty)\times\Omega\times\s^1)}
+\nm{h}_{L^2(\Omega\times\s^1)}+\frac{1}{\e^{1/2}}\nm{\ue^{\l
t}g}_{L^2([0,\infty)\times\Gamma^-)}\bigg),\no
\end{eqnarray}
for any $0\leq\l\leq\l_0^{\ast}$. When $\l_0=0$, we have $\l_0^{\ast}=0$.
\end{theorem}
\begin{proof}
We divide the proof into several steps:\\
\ \\
Step 1: Weak formulation.\\
In the weak formulation (\ref{well-posedness temp 4}), we may take
the test function $\phi=u_{\l}$ to get the energy estimate
\begin{eqnarray}
&&\l\nm{u_{\l}}_{L^2([0,t]\times\Omega\times\s^1)}^2+\half\e\int_0^t\int_{\Gamma}\abs{u_{\l}}^2\ud{\gamma}\\
&&+\half\e^2\nm{u_{\l}(t)}_{L^2(\Omega\times\s^1)}^2-\half\e^2\nm{h}_{L^2(\Omega\times\s^1)}^2+\nm{u_{\l}-\bar
u_{\l}}_{L^2([0,t]\times\Omega\times\s^1)}^2\no\\
&=&\int_0^t\iint_{\Omega\times\s^1}fu_{\l}.\no
\end{eqnarray}
Hence, this naturally implies
\begin{eqnarray}\label{well-posedness temp 5}
&&\half\e\nm{u_{\l}}_{L^2([0,t]\times\Gamma^+)}^2+\half\e^2\nm{u_{\l}(t)}_{L^2(\Omega\times\s^1)}^2+\nm{u_{\l}-\bar
u_{\l}}_{L^2([0,t]\times\Omega\times\s^1)}^2\\
&\leq&\int_0^t\iint_{\Omega\times\s^1}fu_{\l}+\half\e^2\nm{h}_{L^2(\Omega\times\s^1)}^2+\half\e\nm{g}_{L^2([0,t]\times\Gamma^-)}^2.\no
\end{eqnarray}
On the other hand, we can square on both sides of
(\ref{well-posedness temp 3}) to obtain
\begin{eqnarray}\label{well-posedness temp 6}
&&\e^2\tm{\bar u_{\l}}{[0,t]\times\Omega\times\s^1}^2\\
&\leq& C(\Omega)\bigg( \tm{u_{\l}-\bar
u_{\l}}{[0,t]\times\Omega\times\s^1}^2+\tm{f}{[0,t]\times\Omega\times\s^1}^2+\e^2\tm{u_{\l}}{[0,t]\times\Gamma^{+}}^2+\e^2\tm{g}{[0,t]\times\Gamma^-}^2\no\\
&&+\e^4\tm{u_{\l}(t)}{\Omega\times\s^1}^2+\e^4\tm{h}{\Omega\times\s^1}^2\bigg).\nonumber
\end{eqnarray}
Multiplying a sufficiently small constant on both sides of
(\ref{well-posedness temp 6}) and adding it to (\ref{well-posedness
temp 5}) to absorb $\nm{u_{\l}}_{L^2(\Gamma^+)}^2$,
$\tm{u_{\l}(t)}{\Omega\times\s^1}^2$ and $\nm{u_{\l}-\bar
u_{\l}}_{L^2(\Omega\times\s^1)}^2$, we deduce
\begin{eqnarray}
&&\e\nm{u_{\l}}_{L^2([0,t]\times\Gamma^+)}^2+\e^2\nm{u_{\l}(t)}_{L^2(\Omega\times\s^1)}^2+\e^2\nm{\bar
u_{\l}}_{L^2([0,t]\times\Omega\times\s^1)}^2+\nm{u_{\l}-\bar
u_{\l}}_{L^2([0,t]\times\Omega\times\s^1)}^2\\&&\qquad\qquad\qquad\leq
C(\Omega)\bigg(\tm{f}{[0,t]\times\Omega\times\s^1}^2+
\int_0^t\iint_{\Omega\times\s^1}fu_{\l}+\e^2\nm{h}_{L^2(\Omega\times\s^1)}^2+\e\nm{g}_{L^2([0,t]\times\Gamma^-)}^2\bigg).\nonumber
\end{eqnarray}
Hence, we have
\begin{eqnarray}\label{well-posedness temp 7}
&&\e\nm{u_{\l}}_{L^2([0,t]\times\Gamma^+)}^2+\e^2\nm{u_{\l}(t)}_{L^2(\Omega\times\s^1)}^2+\e^2\nm{u_{\l}}_{L^2([0,t]\times\Omega\times\s^1)}^2\\
&\leq& C(\Omega)\bigg(\tm{f}{[0,t]\times\Omega\times\s^1}^2+
\int_0^t\iint_{\Omega\times\s^1}fu_{\l}+\e^2\nm{h}_{L^2(\Omega\times\s^1)}^2+\e\nm{g}_{L^2([0,t]\times\Gamma^-)}^2\bigg).\no
\end{eqnarray}
A simple application of Cauchy's inequality leads to
\begin{eqnarray}
\int_0^t\iint_{\Omega\times\s^1}fu_{\l}\leq\frac{1}{4C\e^2}\tm{f}{[0,t]\times\Omega\times\s^1}^2+C\e^2\tm{u_{\l}}{[0,t]\times\Omega\times\s^1}^2.
\end{eqnarray}
Taking $C$ sufficiently small, we can divide (\ref{well-posedness
temp 7}) by $\e^2$ to obtain
\begin{eqnarray}\label{well-posedness temp 21}
&&\frac{1}{\e}\nm{u_{\l}}_{L^2([0,t]\times\Gamma^+)}^2+\nm{u_{\l}(t)}_{L^2(\Omega\times\s^1)}^2+\nm{u_{\l}}_{L^2([0,t]\times\Omega\times\s^1)}^2\\
&\leq& C(\Omega)\bigg(
\frac{1}{\e^4}\nm{f}_{L^2([0,t]\times\Omega\times\s^1)}^2+\nm{h}_{L^2(\Omega\times\s^1)}^2+\frac{1}{\e}\nm{g}_{L^2([0,t]\times\Gamma^-)}^2\bigg).\no
\end{eqnarray}
\ \\
Step 2: Convergence.\\
Since above estimate does not depend on $\l$, it gives a uniform
estimate for the penalized neutron transport equation
(\ref{well-posedness penalty equation}). Thus, we can extract a
weakly convergent subsequence $u_{\l}\rt u$ as $\l\rt0$. The weak
lower semi-continuity of norms
$\nm{\cdot}_{L^2([0,t]\times\Omega\times\s^1)}$ and
$\nm{\cdot}_{L^2([0,t]\times\Gamma^+)}$ implies $u$ also satisfies
the estimate (\ref{well-posedness temp 21}). Hence, in the weak
formulation (\ref{well-posedness temp 4}), we can take $\l\rt0$ to
deduce that $u$ satisfies equation (\ref{neutron}). Also $u_{\l}-u$
satisfies the equation
\begin{eqnarray}
\\
\left\{
\begin{array}{rcl}
\e^2\dt(u_{\l}-u)+\e\vec w\cdot\nabla_x(u_{\l}-u)+(u_{\l}-u)-(\bar
u_{\l}-\bar u)&=&-\l u_{\l}\ \ \text{in}\ \
\Omega\label{remainder},\\\rule{0ex}{1.0em}
(u_{\l}-u)(0,\vx,\vw)&=&0\ \ \text{in}\ \ \Omega,\\\rule{0ex}{1.0em}
(u_{\l}-u)(\vec x_0,\vec w)&=&0\ \ \text{for}\ \ \vec
x_0\in\p\Omega\ \ \text{and}\ \vw\cdot\vec n<0.\no
\end{array}
\right.
\end{eqnarray}
By a similar argument as above, we can achieve
\begin{eqnarray}
\nm{u_{\l}-u}_{L^2([0,t]\times\Omega\times\s^1)}^2\leq
C(\Omega)\bigg(\frac{\l}{\e^4}\nm{u_{\l}}_{L^2([0,t]\times\Omega\times\s^1)}^2\bigg).
\end{eqnarray}
When $\l\rt0$, the right-hand side approaches zero, which implies
the convergence is actually in the strong sense. The uniqueness
easily follows from the energy estimates.\\
\ \\
Step 3: $L^2$ Decay.\\
Let $v=\ue^{\l t}u$. Then $v$ satisfies the equation
\begin{eqnarray}
\left\{
\begin{array}{rcl}
\e^2\dt v+\e\vw\cdot\nabla_xv+v-\bar v&=&f+\l\e^2v\ \ \text{in}\ \
\Omega,\\\rule{0ex}{1.0em} v(0,\vx,\vw)&=&h(\vx,\vw)\ \ \text{in}\ \
\Omega,\\\rule{0ex}{1.0em} v(\vx_0,\vw)&=&\ue^{\l t}g(t,\vx_0,\vw)\
\ \text{for}\ \ \vx_0\in\p\Omega\ \ \text{and}\ \vw\cdot\vec n<0.
\end{array}
\right.
\end{eqnarray}
Similar to the argument in Step 1, we can obtain
\begin{eqnarray}
&&\frac{1}{\e}\nm{v}_{L^2([0,t]\times\Gamma^+)}^2+\nm{v(t)}_{L^2(\Omega\times\s^1)}^2+\nm{v}_{L^2([0,t]\times\Omega\times\s^1)}^2\\
&\leq& C(\Omega)\bigg( \frac{1}{\e^4}\nm{\ue^{\l
t}f}_{L^2([0,t]\times\Omega\times\s^1)}^2+\frac{1}{\e^4}\nm{\l^2\e^4
v}_{L^2([0,t]\times\Omega\times\s^1)}^2
+\nm{h}_{L^2(\Omega\times\s^1)}^2+\frac{1}{\e}\nm{\ue^{\l
t}g}_{L^2([0,t]\times\Gamma^-)}^2\bigg).\no
\end{eqnarray}
Then when $\l$ is sufficiently small, we have
\begin{eqnarray}
&&\frac{1}{\e}\nm{v}_{L^2([0,t]\times\Gamma^+)}^2+\nm{v(t)}_{L^2(\Omega\times\s^1)}^2+\nm{v}_{L^2([0,t]\times\Omega\times\s^1)}^2\\
&\leq& C(\Omega)\bigg( \frac{1}{\e^4}\nm{\ue^{\l
t}f}_{L^2([0,t]\times\Omega\times\s^1)}^2
+\nm{h}_{L^2(\Omega\times\s^1)}^2+\frac{1}{\e}\nm{\ue^{\l
t}g}_{L^2([0,t]\times\Gamma^-)}^2\bigg),\no
\end{eqnarray}
which implies exponential decay of $u$.

\end{proof}

\subsection{$L^{\infty}$ Estimate}

\begin{theorem}\label{LI estimate}
Assume $\ue^{\l_0 t}f(t,\vx,\vw)\in
L^{\infty}([0,\infty)\times\Omega\times\s^1)$, $h(\vx,\vw)\in
L^{\infty}(\Omega\times\s^1)$ and $\ue^{\l_0 t}g(t,x_0,\vw)\in
L^{\infty}([0,\infty)\times\Gamma^-)$ for some $\l_0>0$. Then for the unsteady neutron
transport equation (\ref{neutron}), there exists $\l_0^{\ast}$ satisfying $0<\l_0^{\ast}\leq\l_0$ and a unique solution
$u(t,\vx,\vw)\in L^{\infty}([0,\infty)\times\Omega\times\s^1)$ satisfying
\begin{eqnarray}
&&\nm{\ue^{\l t}u}_{L^{\infty}([0,\infty)\times\Gamma^+)}+\nm{\ue^{\l t}u(t)}_{L^{\infty}(\Omega\times\s^1)}+\nm{\ue^{\l t}u}_{L^{\infty}([0,\infty)\times\Omega\times\s^1)}\\
&\leq& C(\Omega)\bigg( \frac{1}{\e^{4}}\nm{\ue^{\l
t}f}_{L^{2}([0,\infty)\times\Omega\times\s^1)}
+\frac{1}{\e^{2}}\nm{h}_{L^{2}(\Omega\times\s^1)}+\frac{1}{\e^{5/2}}\nm{\ue^{\l
t}g}_{L^{2}([0,\infty)\times\Gamma^-)}\no\\
&&+\nm{\ue^{\l
t}f}_{L^{\infty}([0,\infty)\times\Omega\times\s^1)}
+\nm{h}_{L^{\infty}(\Omega\times\s^1)}+\nm{\ue^{\l
t}g}_{L^{\infty}([0,\infty)\times\Gamma^-)}\bigg),\no
\end{eqnarray}
for any $0\leq\l\leq\l_0^{\ast}$. When $\l_0=0$, we have $\l_0^{\ast}=0$.
\end{theorem}
\begin{proof}
We divide the proof into several steps to bootstrap an $L^2$ solution to an $L^{\infty}$ solution:\\
\ \\
Step 1: Double Duhamel iterations.\\
The characteristics of the equation (\ref{neutron}) is given by
(\ref{character}). Hence, we can rewrite the equation
(\ref{neutron}) along the characteristics as
\begin{eqnarray}
&&u(t,\vx,\vw)\\
&=&{\bf 1}_{\{t\geq \e^2t_b\}}\bigg( g(t-\e^2t_b,\vx-\e
t_b\vw,\vw)\ue^{-t_b}+\int_{0}^{t_b}(\bar u+f)(t-\e^2(t_b-s),\vx-\e(t_b-s)\vw,\vw)\ue^{-(t_b-s)}\ud{s}\bigg)\no\\
&&+{\bf 1}_{\{t\leq \e^2t_b\}}\bigg( h(\vx-(\e
t\vw)/\e^2,\vw)\ue^{-t/\e^2}+\int_{0}^{t/\e^2}(\bar
u+f)(\e^2s,\vx-\e(t/\e^2-s)\vw,\vw)\ue^{-(t/\e^2-s)}\ud{s}\bigg)\no,
\end{eqnarray}
where the backward exit time $t_b$ is defined as (\ref{exit time}).
For the convenience of analysis, we transform it into a simpler form
\begin{eqnarray}
&&u(t,\vx,\vw)\\
&=&{\bf 1}_{\{t\geq \e^2t_b\}}g(t-\e^2t_b,\vx-\e
t_b\vw,\vw)\ue^{-t_b}+{\bf 1}_{\{t\leq \e^2t_b\}} h(\vx-(\e t\vw)/\e^2,\vw)\ue^{-t/\e^2}\no\\
&&+\int_{t/\e^2-t_b\wedge(t/\e^2)}^{t/\e^2}f(\e^2s,\vx-\e(t/\e^2-s)\vw,\vw)\ue^{-(t/\e^2-s)}\ud{s}\no\\
&&+\frac{1}{2\pi}\int_{t/\e^2-t_b\wedge(t/\e^2)}^{t/\e^2}\bigg(\int_{-\pi}^{\pi}u(\e^2s,\vx-\e(t/\e^2-s)\vw,\vw_t)\ud{\vw_t}\bigg)\ue^{-(t/\e^2-s)}\ud{s}.\no
\end{eqnarray}
Here $a\wedge b$ denotes $\min\{a,b\}$.
Note we have replaced $\bar u$ by the integral of $u$ over the dummy
velocity variable $\vw_t$. For the last term in this formulation, we
apply the Duhamel's principle again to
$u(\e^2s,\vx-\e(t/\e^2-s)\vw,\vw_t)$ and obtain
\begin{eqnarray}\label{well-posedness temp 8}
&&u(t,\vx,\vw)\\
&=&{\bf 1}_{\{t\geq \e^2t_b\}}g(t-\e^2t_b,\vx-\e
t_b\vw,\vw)\ue^{-t_b}+{\bf 1}_{\{t\leq \e^2t_b\}} h(\vx-(\e t\vw)/\e^2,\vw)\ue^{-t/\e^2}\no\\
&&+\int_{t/\e^2-t_b\wedge(t/\e^2)}^{t/\e^2}f(\e^2s,\vx-\e(t/\e^2-s)\vw,\vw)\ue^{-(t/\e^2-s)}\ud{s}\no\\
&&+\frac{1}{2\pi}\int_{t/\e^2-t_b\wedge(t/\e^2)}^{t/\e^2}\bigg(\int_{-\pi}^{\pi}{\bf 1}_{\{s\geq s_b\}}g(\e^2(s-s_b),\vx-\e(t/\e^2-s)\vw,\vw_t)\ue^{-s_b}\ud{\vw_t}\bigg)\ue^{-(t/\e^2-s)}\ud{s}\no\\
&&+\frac{1}{2\pi}\int_{t/\e^2-t_b\wedge(t/\e^2)}^{t/\e^2}\bigg(\int_{-\pi}^{\pi}{\bf 1}_{\{s\geq s_b\}}h(\vx-\e(t/\e^2-s)\vw-\e s\vw_t,\vw_t)\ue^{-s}\ud{\vw_t}\bigg)\ue^{-(t/\e^2-s)}\ud{s}\no\\
&&+\frac{1}{2\pi}\int_{t/\e^2-t_b\wedge(t/\e^2)}^{t/\e^2}
\bigg(\int_{-\pi}^{\pi}\int_{s-s_b\wedge s}^{s}f(\e^2r,\vx-\e(t/\e^2-s)\vw-\e(s-r)\vw_t,\vw_t)\ue^{-(s-r)}\ud{r}\ud{\vw_t}\bigg)\ue^{-(t/\e^2-s)}\ud{s}\no\\
&&+\frac{1}{2\pi}\int_{t/\e^2-t_b\wedge(t/\e^2)}^{t/\e^2}\no\\
&&\bigg(\int_{-\pi}^{\pi}\int_{s-s_b\wedge
s}^{s}\bar u(\e^2r,\vx-\e(t/\e^2-s)\vw-\e(s-r)\vw_t)\ue^{-(s-r)}\ud{r}\ud{\vw_t}\bigg)\ue^{-(t/\e^2-s)}\ud{s}.\no
\end{eqnarray}
where we introduce another dummy velocity variable $\vw_s$ and
\begin{eqnarray}
s_b(\vx,\vw,s,\vw_t)=\inf\{r\geq0: (\vx-\e(t/\e^2-s)\vw-\e
r\vw_t,\vw_t)\in\Gamma^-\}.
\end{eqnarray}
\ \\
Step 2: Estimates of all but the last term in (\ref{well-posedness temp 8}).\\
We can directly estimate as follows:
\begin{eqnarray}\label{im temp 1}
\abs{{\bf 1}_{\{t\geq \e^2t_b\}}g(t-\e^2t_b,\vx-\e
t_b\vw,\vw)\ue^{-t_b}}&\leq&\im{g}{[0,t]\times\Gamma^-},
\end{eqnarray}
\begin{eqnarray}\label{im temp 2}
\abs{{\bf 1}_{\{t\leq \e^2t_b\}} h(\vx-(\e
t\vw)/\e^2,\vw)\ue^{-t/\e^2}} \leq \im{h}{\Omega\times\s^1},
\end{eqnarray}
\begin{eqnarray}\label{im temp 3}
\abs{\int_{t/\e^2-t_b\wedge(t/\e^2)}^{t/\e^2}f(\e^2s,\vx-\e(t/\e^2-s)\vw,\vw)\ue^{-(t/\e^2-s)}\ud{s}}\leq
\im{f}{[0,t]\times\Omega\times\s^1},
\end{eqnarray}
\begin{eqnarray}\label{im temp 4}
&&\frac{1}{2\pi}\int_{t/\e^2-t_b\wedge(t/\e^2)}^{t/\e^2}\bigg(\int_{-\pi}^{\pi}{\bf 1}_{\{s\geq s_b\}}g(\e^2(s-s_b),\vx-\e(t/\e^2-s)\vw,\vw_t)\ue^{-s_b}\ud{\vw_t}\bigg)\ue^{-(t/\e^2-s)}\ud{s}\\
&\leq& \im{g}{[0,t]\times\Gamma^-},\no
\end{eqnarray}
\begin{eqnarray}\label{im temp 7}
&&\abs{\frac{1}{2\pi}\int_{t/\e^2-t_b\wedge(t/\e^2)}^{t/\e^2}\bigg(\int_{-\pi}^{\pi}{\bf 1}_{\{s\geq s_b\}}h(\vx-\e(t/\e^2-s)\vw-\e s\vw_t,\vw_t)\ue^{-s}\ud{\vw_t}\bigg)\ue^{-(t/\e^2-s)}\ud{s}}\\
&\leq& \im{h}{\Omega\times\s^1},\nonumber
\end{eqnarray}
\begin{eqnarray}\label{im temp 8}
\\
&&\abs{\frac{1}{2\pi}\int_{t/\e^2-t_b\wedge(t/\e^2)}^{t/\e^2}
\bigg(\int_{-\pi}^{\pi}\int_{s-s_b\wedge s}^{s}f(\e^2r,\vx-\e(t/\e^2-s)\vw-\e(s-r)\vw_t,\vw_t)\ue^{-(s-r)}\ud{r}\ud{\vw_t}\bigg)\ue^{-(t/\e^2-s)}\ud{s}}\no\\
&\leq& \im{f}{[0,t]\times\Omega\times\s^1}\nonumber.
\end{eqnarray}
\ \\
Step 3: Estimates of the last term in (\ref{well-posedness temp 8}).\\
We can first transform the last term $I$ in (\ref{well-posedness temp 8}) into
\begin{eqnarray}\label{im temp 9}
\
\end{eqnarray}
\begin{eqnarray}
\abs{I}
&\leq&\frac{1}{2\pi}\int_{t/\e^2-t_b\wedge(t/\e^2)}^{t/\e^2}\bigg(\int_{-\pi}^{\pi}\int_{s-s_b\wedge
s}^{s}\abs{\bar u(\e^2r,\vx-\e(t/\e^2-s)\vw-\e(s-r)\vw_t)}\ue^{-(s-r)}\ud{r}\ud{\vw_t}\bigg)\ue^{-(t/\e^2-s)}\ud{s}\no\\
&\leq&\frac{1}{2\pi}\int_{0}^{t_b}\bigg(\int_{-\pi}^{\pi}\int_{0}^{s_b}\no\\
&&\abs{\bar u(\e^2(r^{\ast}+s^{\ast}+t/\e^2-t_b-s_b),\vx-\e(t_b-s^{\ast})\vw-\e(s_b-r^{\ast})\vw_t)}
\ue^{-(s_b-r^{\ast})}\ud{r^{\ast}}\ud{\vw_t}\bigg)\ue^{-(t_b-s^{\ast})}\ud{s^{\ast}}\no,
\end{eqnarray}
by substitution $s\rt s^{\ast}=(s-t/\e^2+t_b)$ and $r\rt r^{\ast}=(r-s+s_b)$.
Now we decompose the right-hand side in (\ref{im temp 9}) as
\begin{eqnarray}
\int_{0}^{t_b}\int_{\s^1}\int_0^{s_b}=\int_{0}^{t_b}\int_{\s^1}\int_{s_b-r^{\ast}\leq\delta}+
\int_{0}^{t_b}\int_{\s^1}\int_{s_b-r^{\ast}\geq\delta}=I_1+I_2,
\end{eqnarray}
for some $\delta>0$. We can estimate $I_1$ directly as
\begin{eqnarray}\label{im temp 5}
I_1
&\leq&\int_{0}^{t_b}\ue^{-(t_b-r^{\ast})}\bigg(\int_{\max\{0,s_b-\delta\}}^{s_b}
\im{u}{[0,t]\times\Omega\times\s^1}\ud{r^{\ast}}\bigg)\ud{s^{\ast}}\leq\delta\im{u}{[0,t]\times\Omega\times\s^1}.
\end{eqnarray}
Then we can bound $I_2$ as
\begin{eqnarray}
\\
I_2&\leq&C\int_{0}^{t_b}\int_{\s^1}\int_{0}^{\max\{0,s_b-\delta\}}\no\\
&&\abs{\bar u(\e^2(r^{\ast}+s^{\ast}+t/\e^2-t_b-s_b),\vx-\e(t_b-s^{\ast})\vw-\e(s_b-r^{\ast})\vw_t)}\ue^{-(s_b-r^{\ast})-(t_b-s^{\ast})}\ud{r^{\ast}}\ud{\vw_t}\ud{s^{\ast}}.\no
\end{eqnarray}
By the definition of $t_b$ and $s_b$, we always have $\e^2(r^{\ast}+s^{\ast}+t/\e^2-t_b-s_b)\in[0,t]$ and
$\vx-\e(t_b-s^{\ast})\vw-\e(s_b-r^{\ast})\vw_t\in\bar\Omega$.
Hence, we may interchange the order of integration and apply
H\"older's inequality to obtain
\begin{eqnarray}\label{well-posedness temp 22}
\
\end{eqnarray}
\begin{eqnarray}
I_2&\leq&C\int_{0}^{t_b}\int_{\s^1}\int_{0}^{\max\{0,s_b-\delta\}}{\bf{1}}_{[0,t]\times\Omega}
(\e^2(r^{\ast}+s^{\ast}+t/\e^2-t_b-s_b),\vx-\e(t_b-s^{\ast})\vw-\e(s_b-r^{\ast})\vw_t)\ue^{-(s_b-r^{\ast})-(t_b-s^{\ast})}\no\\
&&\abs{\bar u(\e^2(r^{\ast}+s^{\ast}+t/\e^2-t_b-s_b),\vx-\e(t_b-s^{\ast})\vw-\e(s_b-r^{\ast})\vw_t)}\ud{r^{\ast}}\ud{\vw_t}\ud{s^{\ast}}\nonumber\\
&\leq&C\bigg(\int_{0}^{t_b}\int_{\s^1}\int_{0}^{\max\{0,s_b-\delta\}}\ue^{-(s_b-r^{\ast})-(t_b-s^{\ast})}\ud{r^{\ast}}\ud{\vw_t}\ud{s^{\ast}}\bigg)^{1/2}\no\\
&&\times \bigg(\int_{0}^{t_b}\int_{\s^1}\int_{0}^{\max\{0,s_b-\delta\}}\no\\
&&{\bf{1}}_{[0,t]\times\Omega}
(\e^2(r^{\ast}+s^{\ast}+t/\e^2-t_b-s_b),\vx-\e(t_b-s^{\ast})\vw-\e(s_b-r^{\ast})\vw_t)\ue^{-(s_b-r^{\ast})-(t_b-s^{\ast})}
\no\\
&&\abs{\bar u^2(\e^2(r^{\ast}+s^{\ast}+t/\e^2-t_b-s_b),\vx-\e(t_b-s^{\ast})\vw-\e(s_b-r^{\ast})\vw_t)}\ud{r^{\ast}}\ud{\vw_t}\ud{s^{\ast}}\bigg)^{1/2}.\nonumber
\end{eqnarray}
Note $\vw_t\in\s^1$, which is essentially a one-dimensional
variable. Thus, we may write it in a new variable $\psi$ as
$\vw_t=(\cos\psi,\sin\psi)$. Then we define the change of variable
$[0,t/\e^2]\times[-\pi,\pi)\times\r\rt [0,t]\times\Omega: (s^{\ast},\psi,r^{\ast})\rt(t',y_1,y_2)=(\e^2(r^{\ast}+s^{\ast}+t/\e^2-t_b-s_b),\vx-\e(t_b-s^{\ast})\vw-\e(s_b-r^{\ast})\vw_t)$, i.e.
\begin{eqnarray}
\left\{
\begin{array}{rcl}
t'&=&\e^2(r^{\ast}+s^{\ast}+t/\e^2-t_b-s_b),\\
y_1&=&x_1-\e(t_b-s^{\ast})w_1-\e(s_b-r^{\ast})\cos\psi,\\
y_2&=&x_2-\e(t_b-s^{\ast})w_2-\e(s_b-r^{\ast})\sin\psi.
\end{array}
\right.
\end{eqnarray}
Therefore, for $s_b-r^{\ast}\geq\delta$, we can directly compute the
Jacobian
\begin{eqnarray}
\\
\abs{\frac{\p{(t',y_1,y_2)}}{\p{(s^{\ast},\psi,r^{\ast})}}}=\abs{\abs{\begin{array}{ccc}
\e^2&0&\e^2\\
\e w_1&\e(s_b-r^{\ast})\sin\psi&\e\cos\psi\\
\e w_2&-\e(s_b-r^{\ast})\cos\psi&\e\sin\psi
\end{array}}}=\e^4(s_b-r^{\ast})\bigg(1-(w_1\cos\psi+w_2\sin\psi)\bigg).\no
\end{eqnarray}
Thus, in order to guarantee the Jacobian is strictly positive, we may further decompose $I_2$ into $I_{2,1}+I_{2,2}$ where in $I_{2,1}$, we have $w_1\cos\psi+w_2\sin\psi>1-\delta$ and in $I_{2,2}$, we have $w_1\cos\psi+w_2\sin\psi\leq1-\delta$. Since $\vw=(w_1,w_2)\in\s^1$, based on trigonometric identity, we obtain
\begin{eqnarray}
\\
I_{2,1}&\leq&C\int_{0}^{t_b}\int_{\vw\cdot\vw_t>1-\delta}\int_{0}^{\max\{0,s_b-\delta\}}\no\\
&&\abs{\bar u(\e^2(r^{\ast}+s^{\ast}+t/\e^2-t_b-s_b),\vx-\e(t_b-s^{\ast})\vw-\e(s_b-r^{\ast})\vw_t)}\ue^{-(s_b-r^{\ast})-(t_b-s^{\ast})}\ud{r^{\ast}}\ud{\vw_t}\ud{s^{\ast}}\no\\
&\leq&\delta\im{u}{[0,t]\times\Omega\times\s^1}.\no
\end{eqnarray}
Hence, we may simplify (\ref{well-posedness temp 22}) as
\begin{eqnarray}
I_{2,2}&\leq&C\bigg(\int_{0}^{t}\int_{\s^1}\int_{\Omega}\frac{1}{\e^4\delta^2}\abs{\bar u^2(t',y)}\ud{\vec y}\ud{t'}\bigg)^{1/2}.
\end{eqnarray}
Then we may further utilize $L^2$
estimate of $u$ in Theorem \ref{LT estimate} to obtain
\begin{eqnarray}\label{im temp 6}
I_{2,2}&\leq&\frac{C}{\e^2\delta}\tm{u}{[0,t]\times\Omega\times\s^1}\\
&\leq&\frac{C(\Omega)}{\delta}\bigg(\frac{1}{\e^{4}}\tm{f}{[0,t]\times\Omega\times\s^1}
+\frac{1}{\e^{2}}\tm{h}{\Omega\times\s^1}+\frac{1}{\e^{5/2}}\tm{g}{[0,t]\times\Gamma^-}\bigg)\nonumber.
\end{eqnarray}
\ \\
Step 4: $L^{\infty}$ estimate.\\
In summary, collecting (\ref{im temp 1}), (\ref{im temp 2}),
(\ref{im temp 3}), (\ref{im temp 4}), (\ref{im temp 7}), (\ref{im
temp 8}), (\ref{im temp 5}) and (\ref{im temp 6}), for fixed
$0<\delta<1$, we have
\begin{eqnarray}
&&\abs{u(t,\vx,\vw)}\\
&\leq& \delta
\im{u}{[0,t]\times\Omega\times\s^1}+\frac{C(\Omega)}{\delta}\bigg(\frac{1}{\e^{4}}\tm{f}{[0,t]\times\Omega\times\s^1}
+\frac{1}{\e^{2}}\tm{h}{\Omega\times\s^1}+\frac{1}{\e^{5/2}}\tm{g}{[0,t]\times\Gamma^-}\bigg)\no\\
&&+\bigg(\im{f}{[0,t]\times\Omega\times\s^1}
+\im{h}{\Omega\times\s^1}+\im{g}{[0,t]\times\Gamma^-}\bigg).\no
\end{eqnarray}
Then we may take $0<\delta\leq1/2$ to obtain
\begin{eqnarray}
&&\abs{u(t,\vx,\vw)}\\
&\leq&
\half\im{u}{[0,t]\times\Omega\times\s^1}+C(\Omega)\bigg(\frac{1}{\e^{4}}\tm{f}{[0,t]\times\Omega\times\s^1}
+\frac{1}{\e^{2}}\tm{h}{\Omega\times\s^1}+\frac{1}{\e^{5/2}}\tm{g}{[0,t]\times\Gamma^-}\bigg)\no\\
&&+\bigg(\im{f}{[0,t]\times\Omega\times\s^1}
+\im{h}{\Omega\times\s^1}+\im{g}{[0,t]\times\Gamma^-}\bigg).\no
\end{eqnarray}
Taking supremum of $u$ over all $(t,\vx,\vw)$, we have
\begin{eqnarray}
&&\im{u}{[0,t]\times\Omega\times\s^1}\\
&\leq&
\half\im{u}{[0,t]\times\Omega\times\s^1}+C(\Omega)\bigg(\frac{1}{\e^{4}}\tm{f}{[0,t]\times\Omega\times\s^1}
+\frac{1}{\e^{2}}\tm{h}{\Omega\times\s^1}+\frac{1}{\e^{5/2}}\tm{g}{[0,t]\times\Gamma^-}\bigg)\no\\
&&+\bigg(\im{f}{[0,t]\times\Omega\times\s^1}
+\im{h}{\Omega\times\s^1}+\im{g}{[0,t]\times\Gamma^-}\bigg).\no
\end{eqnarray}
Finally, absorbing $\im{u}{[0,t]\times\Omega\times\s^1}$, we get
\begin{eqnarray}
\im{u}{[0,t]\times\Omega\times\s^1}&\leq&
C(\Omega)\bigg(\frac{1}{\e^{4}}\tm{f}{[0,t]\times\Omega\times\s^1}
+\frac{1}{\e^{2}}\tm{h}{\Omega\times\s^1}+\frac{1}{\e^{5/2}}\tm{g}{[0,t]\times\Gamma^-}\\
&&+\im{f}{[0,t]\times\Omega\times\s^1}
+\im{h}{\Omega\times\s^1}+\im{g}{[0,t]\times\Gamma^-}\bigg).\no
\end{eqnarray}
\ \\
Step 5: $L^{\infty}$ Decay.\\
Let $v=\ue^{\l t}u$. Then $v$ satisfies the equation
\begin{eqnarray}
\left\{
\begin{array}{rcl}
\e^2\dt v+\e\vw\cdot\nabla_xv+(1-\l\e^2)v-\bar v&=&f\ \ \text{in}\ \
\Omega,\\\rule{0ex}{1.0em} v(0,\vx,\vw)&=&h(\vx,\vw)\ \ \text{in}\ \
\Omega,\\\rule{0ex}{1.0em} v(\vx_0,\vw)&=&\ue^{\l t}g(t,\vx_0,\vw)\
\ \text{for}\ \ \vx_0\in\p\Omega\ \ \text{and}\ \vw\cdot\vec n<0.
\end{array}
\right.
\end{eqnarray}
By a similar argument as in Step 3 and Step 4, combined with the $L^2$ decay, we can
finally show the desired estimate.
\end{proof}

\subsection{Maximum Principle}

\begin{theorem}\label{maximum principle}
When $f=0$, the solution $u(t,\vx,\vw)$ to the unsteady neutron transport equation (\ref{neutron}) satisfies the maximum principle, i.e.
\begin{eqnarray}
\inf\{g(t,\vx_0,\vw), h(\vx,\vw)\}\leq u(t,\vx,\vw)\leq \sup\{g(t,\vx_0,\vw), h(\vx,\vw)\}.
\end{eqnarray}
\end{theorem}
\begin{proof}
We claim that it suffices to show $u(t,\vx,\vw)\leq 0$ whenever $g(t,\vx_0,\vw)\leq 0$ and $h(\vx,\vw)\leq0$. Suppose the claim is justified. Then define
\begin{eqnarray}
m&=&\inf\{g(t,\vx_0,\vw), h(\vx,\vw)\},\\
M&=&\sup\{g(t,\vx_0,\vw), h(\vx,\vw)\}.
\end{eqnarray}
We have $u_1=u-M$ satisfies the equation
\begin{eqnarray}
\left\{ \begin{array}{rcl} \e^2\dt u_1+\e\vec w\cdot\nabla_xu_1+u_1-\bar
u_1&=&0\ \ \text{in}\ \
[0,\infty)\times\Omega,\\\rule{0ex}{1.0em}
u_1(0,\vx,\vw)&=&h(\vx,\vw)-M\ \ \text{in}\ \ \Omega\\\rule{0ex}{1.0em}
u_1(t,\vec x_0,\vec w)&=&g(t,\vec x_0,\vec w)-M\ \ \text{for}\ \ \vw\cdot\vec n<0\ \
\text{and}\ \ \vx_0\in\p\Omega,
\end{array}
\right.
\end{eqnarray}
Hence, $h-M\leq 0$ and $g-M\leq0$ implies $u_1\leq 0$, which is actually $u\leq M$. Similarly, we have $u_2=m-u$ satisfies the equation
\begin{eqnarray}
\left\{ \begin{array}{rcl} \e^2\dt u_2+\e\vec w\cdot\nabla_xu_2+u_2-\bar
u_2&=&0\ \ \text{in}\ \
[0,\infty)\times\Omega,\\\rule{0ex}{1.0em}
u_2(0,\vx,\vw)&=&m-h(\vx,\vw)\ \ \text{in}\ \ \Omega\\\rule{0ex}{1.0em}
u_2(t,\vec x_0,\vec w)&=&m-g(t,\vec x_0,\vec w)\ \ \text{for}\ \ \vw\cdot\vec n<0\ \
\text{and}\ \ \vx_0\in\p\Omega,
\end{array}
\right.
\end{eqnarray}
Hence, $m-h\leq 0$ and $m-g\leq0$ implies $u_2\leq 0$, which is actually $u\geq m$. Therefore, the maximum principle is established.

We now prove the claim that if $g(t,\vx_0,\vw)\leq 0$ and $h(\vx,\vw)\leq0$, we have $u(t,\vx,\vw)\leq 0$. We first consider the penalized neutron transport equation
\begin{eqnarray}
\left\{ \begin{array}{rcl} \l u_{\l}+\e^2\dt
u_{\l}+\e\vw\cdot\nabla_xu_{\l}+u_{\l}-\bar u_{\l}&=& 0\
\ \text{in}\ \ [0,\infty)\times\Omega,\\\rule{0ex}{1.0em} u_{\l}(0,\vx,\vw)&=&h(\vx,\vw)\ \
\text{in}\ \ \Omega\\\rule{0ex}{1.0em} u_{\l}(t,\vx_0,\vec
w)&=&g(t,\vx_0,\vw)\ \ \text{for}\ \ \vx_0\in\p\Omega\ \ \text{and}\
\vw\cdot\vec n<0.
\end{array}
\right.
\end{eqnarray}
with $\l>0$. Based on Lemma \ref{well-posedness lemma 2}, there exists a solution $u_{\l}(t,\vx,\vw)\in
L^{\infty}([0,\infty)\times\Omega\times\s^1)$. We use the notation in the proof of Lemma \ref{well-posedness lemma 2}.
Define an approximating sequence $\{u_{\l}^k\}_{k=0}^{\infty}$,
where $u_{\l}^0=0$ and
\begin{eqnarray}
\left\{
\begin{array}{rcl}
\l u_{\l}^{k}+\e^2\dt
u_{\l}^k+\e\vw\cdot\nabla_xu_{\l}^k+u_{\l}^k-\bar
u_{\l}^{k-1}&=&0\ \ \text{in}\ \
[0,\infty)\times\Omega,\\\rule{0ex}{1.0em}
u_{\l}^k(0,\vx,\vw)&=&h(\vx,\vw)\ \ \text{in}\ \
\Omega\\\rule{0ex}{1.0em} u_{\l}^k(t,\vx_0,\vw)&=&g(t,\vx_0,\vw)\ \
\text{for}\ \ \vx_0\in\p\Omega\ \ \text{and}\ \vw\cdot\vec n<0.
\end{array}
\right.
\end{eqnarray}
By Lemma \ref{well-posedness lemma 1}, this sequence is well-defined
and $\im{u_{\l}^k}{[0,\infty)\times\Omega\times\s^1}<\infty$. Then we rewrite equation
(\ref{penalty temp 1}) along the characteristics as
\begin{eqnarray}
\
\end{eqnarray}
\begin{eqnarray}
&&u_{\l}^k(t,\vx,\vw)\no\\
&=&{\bf 1}_{\{t\geq \e^2t_b\}}\bigg( g(t-\e^2t_b,\vx-\e
t_b\vw,\vw)\ue^{-(1+\l)t_b}+\int_{0}^{t_b}\bar u_{\l}^{k-1}(t-\e^2(t_b-s),\vx-\e(t_b-s)\vw,\vw)\ue^{-(1+\l)(t_b-s)}\ud{s}\bigg)\no\\
&&+{\bf 1}_{\{t\leq \e^2t_b\}}\bigg( h(\vx-(\e
t\vw)/\e^2,\vw)\ue^{-(1+\l)t/\e^2}+\int_{0}^{t/\e^2}\bar
u_{\l}^{k-1}(\e^2s,\vx-\e(t/\e^2-s)\vw,\vw)\ue^{-(1+\l)(t/\e^2-s)}\ud{s}\bigg)\no,
\end{eqnarray}
where
\begin{equation}
t_b(\vx,\vw)=\inf\{s\geq0: (\vx-\e s\vw,\vw)\in\Gamma^-\}.
\end{equation}
Since $u_{\l}^k(t,\vx,\vw)\leq 0$ naturally implies $\bar u_{\l}^k(t,\vx)\leq 0$, we naturally have $u_{\l}^k(t,\vx,\vw)\leq 0$ when $g(t,\vx_0,\vw)\leq 0$ and $h(\vx,\vw)\leq0$. In the proof of Lemma \ref{well-posedness lemma 2}, we have shown $u_{\l}^k\rt u_{\l}$ in $L^{\infty}$ as $k\rt\infty$. Therefore, we have $u_{\l}(t,\vx,\vw)\leq 0$. Based on the proof of Lemma \ref{LT estimate}, we know $u_{\l}\rt u$ in $L^2$ as $\l\rt0$, where $u$ is the solution of the equation (\ref{neutron}). Then we naturally obtain $u\leq 0$. Also, this is the unique solution to the equation (\ref{neutron}). This justifies the claim and completes the proof.
\end{proof}
Theorem \ref{maximum principle} naturally leads to the $L^{\infty}$ estimate of the equation (\ref{neutron}).
\begin{corollary}\label{wellposedness estimate 2}
Assume $h(\vx,\vw)\in
L^{\infty}(\Omega\times\s^1)$ and $g(t,x_0,\vw)\in
L^{\infty}([0,\infty)\times\Gamma^-)$. Then for the unsteady neutron
transport equation (\ref{neutron}) with $f=0$, there exists a unique solution
$u(t,\vx,\vw)\in L^{\infty}([0,\infty)\times\Omega\times\s^1)$ satisfying
\begin{eqnarray}
&&\nm{u}_{L^{\infty}([0,\infty)\times\Gamma^+)}+\nm{u(t)}_{L^{\infty}(\Omega\times\s^1)}+\nm{u}_{L^{\infty}([0,\infty)\times\Omega\times\s^1)}
\leq\nm{h}_{L^{\infty}(\Omega\times\s^1)}+\nm{g}_{L^{\infty}([0,\infty)\times\Gamma^-)}.
\end{eqnarray}
\end{corollary}

\subsection{Well-posedness of Transport Equation}

Combining the results in Theorem \ref{LI estimate} and Corollary \ref{wellposedness estimate 2}, we can show an improved $L^{\infty}$ estimate of the equation (\ref{neutron}).
\begin{theorem}\label{improved LI estimate}
Assume $f(t,\vx,\vw)\in
L^{\infty}([0,\infty)\times\Omega\times\s^1)$, $h(\vx,\vw)\in
L^{\infty}(\Omega\times\s^1)$ and $g(t,x_0,\vw)\in
L^{\infty}([0,\infty)\times\Gamma^-)$. Then for the unsteady neutron
transport equation (\ref{neutron}), there exists a unique solution
$u(t,\vx,\vw)\in L^{\infty}([0,\infty)\times\Omega\times\s^1)$ satisfying
\begin{eqnarray}
&&\nm{u}_{L^{\infty}([0,\infty)\times\Gamma^+)}+\nm{u(t)}_{L^{\infty}(\Omega\times\s^1)}+\nm{u}_{L^{\infty}([0,\infty)\times\Omega\times\s^1)}\\
&\leq& C(\Omega)\bigg(\frac{1}{\e^{4}}\nm{f}_{L^{2}([0,\infty)\times\Omega\times\s^1)}+\nm{f}_{L^{\infty}([0,\infty)\times\Omega\times\s^1)}\bigg)
+\nm{h}_{L^{\infty}(\Omega\times\s^1)}+\nm{g}_{L^{\infty}([0,\infty)\times\Gamma^-)}.\no
\end{eqnarray}
\end{theorem}
\begin{proof}
Since the equation (\ref{neutron}) is a linear equation, then we can utilize the superposition property, i.e. we can separate the solution $u=u_1+u_2$ where
$u_1$ satisfies the equation
\begin{eqnarray}\label{improved temp 1}
\left\{ \begin{array}{rcl} \e^2\dt u_1+\e\vec w\cdot\nabla_xu_1+u_1-\bar
u_1&=&0\ \ \text{in}\ \
[0,\infty)\times\Omega,\\\rule{0ex}{1.0em}
u_1(0,\vx,\vw)&=&h(\vx,\vw)\ \ \text{in}\ \ \Omega\\\rule{0ex}{1.0em}
u_1(t,\vec x_0,\vec w)&=&g(t,\vec x_0,\vec w)\ \ \text{for}\ \ \vw\cdot\vec n<0\ \
\text{and}\ \ \vx_0\in\p\Omega,
\end{array}
\right.
\end{eqnarray}
and $u_2$ satisfies the equation
\begin{eqnarray}\label{improved temp 2}
\left\{ \begin{array}{rcl} \e^2\dt u_2+\e\vec w\cdot\nabla_xu_2+u_2-\bar
u_2&=&f(t,\vx,\vw)\ \ \text{in}\ \
[0,\infty)\times\Omega,\\\rule{0ex}{1.0em}
u_2(0,\vx,\vw)&=&0\ \ \text{in}\ \ \Omega\\\rule{0ex}{1.0em}
u_2(t,\vec x_0,\vec w)&=&0\ \ \text{for}\ \ \vw\cdot\vec n<0\ \
\text{and}\ \ \vx_0\in\p\Omega,
\end{array}
\right.
\end{eqnarray}
Note that the data in (\ref{improved temp 1}) and (\ref{improved temp 2}) satisfy the compatibility condition (\ref{compatibility condition}). Therefore, we can apply the previous results in this section. Corollary \ref{wellposedness estimate 2} yields
\begin{eqnarray}\label{improved temp 3}
&&\nm{u_1}_{L^{\infty}([0,\infty)\times\Gamma^+)}+\nm{u_1(t)}_{L^{\infty}(\Omega\times\s^1)}+\nm{u_1}_{L^{\infty}([0,\infty)\times\Omega\times\s^1)}
\leq\nm{h}_{L^{\infty}(\Omega\times\s^1)}+\nm{g}_{L^{\infty}([0,\infty)\times\Gamma^-)}.
\end{eqnarray}
Also, Theorem \ref{LI estimate} leads to
\begin{eqnarray}\label{improved temp 4}
&&\nm{u_2}_{L^{\infty}([0,\infty)\times\Gamma^+)}+\nm{u_2(t)}_{L^{\infty}(\Omega\times\s^1)}+\nm{u_2}_{L^{\infty}([0,\infty)\times\Omega\times\s^1)}\\
&\leq& C(\Omega)\bigg( \frac{1}{\e^{4}}\nm{f}_{L^{2}([0,\infty)\times\Omega\times\s^1)}+\nm{f}_{L^{\infty}([0,\infty)\times\Omega\times\s^1)}
\bigg).\no
\end{eqnarray}
Combining (\ref{improved temp 3}) and (\ref{improved temp 4}), we have the desired result.
\end{proof}
Finally, we can apply Theorem \ref{improved LI estimate} to the equation (\ref{transport}) and obtain Theorem \ref{main 1}.
\begin{theorem}\label{well-posedness 2}
Assume $g(t,x_0,\vw)\in L^{\infty}([0,\infty)\times\Gamma^-)$ and
$h(\vx,\vw)\in L^{\infty}(\Omega\times\s^1)$. Then for the unsteady
neutron transport equation (\ref{transport}), there exists a unique
solution $u^{\e}(t,\vx,\vw)\in
L^{\infty}([0,\infty)\times\Omega\times\s^1)$ satisfying
\begin{eqnarray}
\im{u^{\e}}{[0,\infty)\times\Omega\times\s^1}\leq
\nm{h}_{L^{\infty}(\Omega\times\s^1)}+\nm{g}_{L^{\infty}([0,\infty)\times\Gamma^-)}.
\end{eqnarray}
\end{theorem}
%
%

\section{Asymptotic Analysis}

In this section, we construct the asymptotic expansion of the equation (\ref{transport}).

\subsection{Discussion of Compatibility Condition}

The initial and boundary data satisfy the
compatibility condition
\begin{eqnarray}
h(\vx_0,\vw)=g(0,\vx_0,\vw)\ \ \text{for}\ \ \vw\cdot\vec n<0.
\end{eqnarray}
Then in the half-space $\vw\cdot\vec n<0$ at $(0,\vx_0,\vw)$, the equation
\begin{eqnarray}
\e^2\dt u^{\e}+\e \vw\cdot\nabla_x u^{\e}+u^{\e}-\bar
u^{\e}&=&0,
\end{eqnarray}
is valid, which implies
\begin{eqnarray}\label{comp 1}
\e^2\dt g(0,\vx_0,\vw)+\e \vw\cdot\nabla_xh(\vx_0,\vw)+h(\vx_0,\vw)-\bar
h(\vx_0)&=&0.
\end{eqnarray}
In order to show the diffusive limit, the condition (\ref{comp 1}) holds for arbitrary $\e$. Since $g$ and $h$ are all independent of $\e$, we must have for $\vw\cdot\vec n<0$,
\begin{eqnarray}
\dt g(0,\vx_0,\vw)&=&0,\\
\vw\cdot\nabla_xh(\vx_0,\vw)&=&0,\\
h(\vx_0,\vw)-\bar
h(\vx_0)&=&0.\label{comp 2}
\end{eqnarray}
The relation (\ref{comp 2}) implies the improved compatibility condition
\begin{eqnarray}\label{improved compatibility condition}
h(\vx_0,\vw)=g(0,\vx_0,\vw)=C_0\ \ \text{for}\ \ \vw\cdot\vec n<0,
\end{eqnarray}
for some constant $C_0$. This fact is of great importance in the following analysis.

\subsection{Interior Expansion}

We define the interior expansion as follows:
\begin{eqnarray}\label{interior expansion}
\u(t,\vx,\vw)\sim\sum_{k=0}^{\infty}\e^k\u_k(t,\vx,\vw),
\end{eqnarray}
where $\u_k$ can be defined by comparing the order of $\e$ via
plugging (\ref{interior expansion}) into the equation
(\ref{transport}). Thus, we have
\begin{eqnarray}
\u_0-\bu_0&=&0,\label{expansion temp 1}\\
\u_1-\bu_1&=&-\vw\cdot\nx\u_0,\label{expansion temp 2}\\
\u_2-\bu_2&=&-\vw\cdot\nx\u_1-\dt\u_0,\label{expansion temp 3}\\
\ldots\nonumber\\
\u_k-\bu_k&=&-\vw\cdot\nx\u_{k-1}-\dt\u_{k-2}.
\end{eqnarray}
\ \\
The following analysis reveals the equation satisfied by
$\u_k$:\\
Plugging (\ref{expansion temp 1}) into (\ref{expansion temp 2}), we
obtain
\begin{eqnarray}
\u_1=\bu_1-\vw\cdot\nx\bu_0.\label{expansion temp 4}
\end{eqnarray}
Plugging (\ref{expansion temp 4}) into (\ref{expansion temp 3}), we
get
\begin{eqnarray}\label{expansion temp 13}
\u_2-\bu_2+\dt\u_0=-\vw\cdot\nx(\bu_1-\vw\cdot\nx\bu_0)=-\vw\cdot\nx\bu_1+\abs{\vw}^2\Delta_x\bu_0+2w_1w_2\p_{x_1x_2}\bu_0.
\end{eqnarray}
Integrating (\ref{expansion temp 13}) over $\vw\in\s^1$, we achieve
the final form
\begin{eqnarray}
\dt\bu_0-\Delta_x\bu_0=0,
\end{eqnarray}
which further implies $\u_0(t,\vx,\vw)$ satisfies the equation
\begin{eqnarray}\label{interior 1}
\left\{
\begin{array}{rcl}
\u_0&=&\bu_0,\\
\dt\bu_0-\Delta_x\bu_0&=&0.
\end{array}
\right.
\end{eqnarray}
Similarly, we can derive $\u_1(t,\vx,\vw)$ satisfies
\begin{eqnarray}\label{interior 2}
\left\{
\begin{array}{rcl}
\u_1&=&\bu_1-\vw\cdot\nx\u_0,\\
\dt\bu_1-\Delta_x\bu_1&=&0,
\end{array}
\right.
\end{eqnarray}
and $\u_k(t,\vx,\vw)$ for $k\geq2$ satisfies
\begin{eqnarray}\label{interior 3}
\left\{
\begin{array}{rcl}
\u_k&=&\bu_k-\vw\cdot\nx\u_{k-1}-\dt\u_{k-2},\\
\dt\bu_k-\Delta_x\bu_k&=&0.
\end{array}
\right.
\end{eqnarray}
Note that in order to determine $\u_k$, we need to define the initial data and boundary data.

\subsection{Initial Layer Expansion}

In order to determine the initial condition for $\u_k$, we need to define the initial layer expansion. Hence, we need a substitution:\\
\ \\
Temporal Substitution:\\
We define the stretched variable $\tau$ by making the
scaling transform for $u^{\e}(t)\rt u^{\e}(\tau)$
with $\tau\in [0,\infty)$ as
\begin{eqnarray}\label{substitution 0}
\tau&=&\frac{t}{\e^2},
\end{eqnarray}
which implies
\begin{eqnarray}
\frac{\p u^{\e}}{\p t}=\frac{1}{\e^2}\frac{\p u^{\e}}{\p\tau}.
\end{eqnarray}
In this new variable,
equation (\ref{transport}) can be rewritten as
\begin{eqnarray}\label{initial temp}
\left\{ \begin{array}{l}\displaystyle \p_{\tau}u^{\e}+\e\vw\cdot\nabla_xu^{\e}+u^{\e}-\bar u^{\e}=0,\\\rule{0ex}{1.0em}
u^{\e}(0,\vx,\vw)=h(\vx,\vw),\\\rule{0ex}{1.0em}
u^{\e}(\tau,\vx_0,\vw)=g(\tau,\vx_0,\vw)\ \ \text{for}\ \
\vw\cdot\vec n<0.
\end{array}
\right.
\end{eqnarray}
We define the initial layer expansion as follows:
\begin{eqnarray}\label{initial layer expansion}
\ub_I(\tau,\vx,\vw)\sim\sum_{k=0}^{\infty}\e^k\ub_{I,k}(\tau,\vx,\vw),
\end{eqnarray}
where $\ub_{I,k}$ can be determined by comparing the order of $\e$ via
plugging (\ref{initial layer expansion}) into the equation
(\ref{initial temp}). Thus, we
have
\begin{eqnarray}
\p_{\tau}\ub_{I,0}+\ub_{I,0}-\bub_{I,0}&=&0,\label{initial expansion 1}\\
\p_{\tau}\ub_{I,1}+\ub_{I,1}-\bub_{I,1}&=&-\vw\cdot\nabla_x\ub_{I,0},\label{initial expansion 2}\\
\p_{\tau}\ub_{I,2}+\ub_{I,2}-\bub_{I,2}&=&-\vw\cdot\nabla_x\ub_{I,1},\label{initial expansion 3}\\
\ldots\no\\
\p_{\tau}\ub_{I,k}+\ub_{I,k}-\bub_{I,k}&=&-\vw\cdot\nabla_x\ub_{I,k-1}.\label{initial expansion 4}
\end{eqnarray}
\ \\
The following analysis reveals the equation satisfied by
$\ub_{I,k}$:\\
Integrate (\ref{initial expansion 1}) over $\vw\in\s^1$, we have
\begin{eqnarray}
\p_{\tau}\bub_{I,0}=0.
\end{eqnarray}
which further implies
\begin{eqnarray}
\bub_{I,0}(\tau,\vx)=\bub_{I,0}(0,\vx).
\end{eqnarray}
Therefore, from (\ref{initial expansion 1}), we can deduce
\begin{eqnarray}
\ub_{I,0}(\tau,\vx,\vw)&=&\ue^{-\tau}\ub_{I,0}(0,\vx,\vw)+\int_0^{\tau}\bub_{I,0}(s,\vx)\ue^{s-\tau}\ud{s}\\
&=&\ue^{-\tau}\ub_{I,0}(0,\vx,\vw)+(1-\ue^{-\tau})\bub_{I,0}(0,\vx).\no
\end{eqnarray}
This means we have
\begin{eqnarray}
\left\{
\begin{array}{rcl}
\p_{\tau}\bub_{I,0}&=&0,\\\rule{0ex}{1.0em}
\ub_{I,0}(\tau,\vx,\vw)&=&\ue^{-\tau}\ub_{I,0}(0,\vx,\vw)+(1-\ue^{-\tau})\bub_{I,0}(0,\vx).
\end{array}
\right.
\end{eqnarray}
Similarly, we can derive $\ub_{I,k}(\tau,\vx,\vw)$ for $k\geq1$ satisfies
\begin{eqnarray}
\left\{
\begin{array}{rcl}
\p_{\tau}\bub_{I,k}&=&-\displaystyle\int_{\s^1}\bigg(\vw\cdot\nabla_x\ub_{I,k-1}\bigg)\ud{\vw},\\\rule{0ex}{1.5em}
\ub_{I,k}(\tau,\vx,\vw)&=&\ue^{-\tau}\ub_{I,k}(0,\vx,\vw)+\displaystyle\int_0^{\tau}\bigg(\bub_{I,k}-\vw\cdot\nabla_x\ub_{I,k-1}\bigg)(s,\vx,\vw)\ue^{s-\tau}\ud{s}.
\end{array}
\right.
\end{eqnarray}

\subsection{Boundary Layer Expansion with Geometric Correction}

In order to determine the boundary condition for $\u_k$, we need to define the boundary layer expansion. Hence, we need several substitutions:\\
\ \\
Spacial Substitution 1:\\
We consider the substitution into quasi-polar coordinates
$u^{\e}(x_1,x_2)\rt u^{\e}(\mu,\theta)$ with $(\mu,\theta)\in
[0,1)\times[-\pi,\pi)$ defined as
\begin{eqnarray}\label{substitution 1}
\left\{
\begin{array}{rcl}
x_1&=&(1-\mu)\cos\theta,\\
x_2&=&(1-\mu)\sin\theta.
\end{array}
\right.
\end{eqnarray}
Here $\mu$ denotes the distance to the boundary $\p\Omega$ and
$\theta$ is the space angular variable. In these new variables,
equation (\ref{transport}) can be rewritten as
\begin{eqnarray}
\\
\left\{ \begin{array}{l}\displaystyle \e^2\frac{\p u^{\e}}{\p
t}-\e\bigg(w_1\cos\theta+w_2\sin\theta\bigg)\frac{\p
u^{\e}}{\p\mu}-\frac{\e}{1-\mu}\bigg(w_1\sin\theta-w_2\cos\theta\bigg)\frac{\p
u^{\e}}{\p\theta}+u^{\e}-\frac{1}{2\pi}\int_{\s^1}u^{\e}\ud{\vw}=0,\\\rule{0ex}{1.0em}
u^{\e}(0,\mu,\theta,w_1,w_2)=h(\mu,\theta,w_1,w_2),\\\rule{0ex}{1.0em}
u^{\e}(t,0,\theta,w_1,w_2)=g(t,\theta,w_1,w_2)\ \ \text{for}\ \
w_1\cos\theta+w_2\sin\theta<0.\no
\end{array}
\right.
\end{eqnarray}
\ \\
Spacial Substitution 2:\\
We further define the stretched variable $\eta$ by making the
scaling transform for $u^{\e}(\mu,\theta)\rt u^{\e}(\eta,\theta)$
with $(\eta,\theta)\in [0,1/\e)\times[-\pi,\pi)$ as
\begin{eqnarray}\label{substitution 2}
\left\{
\begin{array}{rcl}
\eta&=&\dfrac{\mu}{\e},\\\rule{0ex}{1.5em}
\theta&=&\theta,
\end{array}
\right.
\end{eqnarray}
which implies
\begin{eqnarray}
\frac{\p u^{\e}}{\p\mu}=\frac{1}{\e}\frac{\p u^{\e}}{\p\eta}.
\end{eqnarray}
Then equation (\ref{transport}) is transformed into
\begin{eqnarray}
\\
\left\{ \begin{array}{l}\displaystyle \e^2\frac{\p u^{\e}}{\p
t}-\bigg(w_1\cos\theta+w_2\sin\theta\bigg)\frac{\p
u^{\e}}{\p\eta}-\frac{\e}{1-\e\eta}\bigg(w_1\sin\theta-w_2\cos\theta\bigg)\frac{\p
u^{\e}}{\p\theta}+u^{\e}-\frac{1}{2\pi}\int_{\s^1}u^{\e}\ud{\vw}=0,\\\rule{0ex}{1.0em}
u^{\e}(0,\eta,\theta,\vw)=h(\eta,\theta,w_1,w_2),\\\rule{0ex}{1.0em}
u^{\e}(t,0,\theta,w_1,w_2)=g(t,\theta,w_1,w_2)\ \ \text{for}\ \
w_1\cos\theta+w_2\sin\theta<0.\no
\end{array}
\right.
\end{eqnarray}
\ \\
Spacial Substitution 3:\\
Define the velocity substitution for $u^{\e}(w_1,w_2)\rt
u^{\e}(\xi)$ with $\xi\in [-\pi,\pi)$ as
\begin{eqnarray}\label{substitution 3}
\left\{
\begin{array}{rcl}
w_1&=&-\sin\xi,\\
w_2&=&-\cos\xi.
\end{array}
\right.
\end{eqnarray}
Here $\xi$ denotes the velocity angular variable. We have the
succinct form for (\ref{transport}) as
\begin{eqnarray}\label{classical temp}
\left\{ \begin{array}{l}\displaystyle \e^2\frac{\p u^{\e}}{\p
t}+\sin(\theta+\xi)\frac{\p
u^{\e}}{\p\eta}-\frac{\e}{1-\e\eta}\cos(\theta+\xi)\frac{\p
u^{\e}}{\p\theta}+u^{\e}-\frac{1}{2\pi}\int_{-\pi}^{\pi}u^{\e}\ud{\xi}=0,\\\rule{0ex}{1.0em}
u^{\e}(0,\eta,\theta,\xi)=h(\eta,\theta,\xi),\\\rule{0ex}{1.0em}
u^{\e}(t,0,\theta,\xi)=g(t,\theta,\xi)\ \ \text{for}\ \
\sin(\theta+\xi)>0.
\end{array}
\right.
\end{eqnarray}
\ \\
Spacial Substitution 4:\\
We make the rotation substitution for $u^{\e}(\xi)\rt u^{\e}(\phi)$
with $\phi\in [-\pi,\pi)$ as
\begin{eqnarray}\label{substitution 4}
\begin{array}{rcl}
\phi&=&\theta+\xi,
\end{array}
\end{eqnarray}
and transform the equation (\ref{transport}) into
\begin{eqnarray}\label{transport temp}
\left\{ \begin{array}{l}\displaystyle \e^2\dfrac{\p u^{\e}}{\p
t}+\sin\phi\frac{\p
u^{\e}}{\p\eta}-\frac{\e}{1-\e\eta}\cos\phi\bigg(\frac{\p
u^{\e}}{\p\phi}+\frac{\p
u^{\e}}{\p\theta}\bigg)+u^{\e}-\frac{1}{2\pi}\int_{-\pi}^{\pi}u^{\e}\ud{\phi}=0,\\\rule{0ex}{1.0em}
u^{\e}(0,\eta,\theta,\phi)=h(\eta,\theta,\phi),\\\rule{0ex}{1.0em}
u^{\e}(t,0,\theta,\phi)=g(t,\theta,\phi)\ \ \text{for}\ \
\sin\phi>0.
\end{array}
\right.
\end{eqnarray}
We define the boundary layer expansion with geometric correction as follows:
\begin{eqnarray}\label{boundary layer expansion}
\ub_B(t,\eta,\theta,\phi)\sim\sum_{k=0}^{\infty}\e^k\ub_{B,k}(t,\eta,\theta,\phi),
\end{eqnarray}
where $\ub_{B,k}$ can be determined by comparing the order of $\e$ via
plugging (\ref{boundary layer expansion}) into the equation
(\ref{transport temp}). Following the idea in \cite{AA003}, in a neighborhood of the boundary, we
require
\begin{eqnarray}
\sin\phi\frac{\p \ub_{B,0}}{\p\eta}-\frac{\e}{1-\e\eta}\cos\phi\frac{\p
\ub_{B,0}}{\p\phi}+\ub_{B,0}-\bub_{B,0}&=&0,\label{expansion temp 5}\\
\sin\phi\frac{\p \ub_{B,1}}{\p\eta}-\frac{\e}{1-\e\eta}\cos\phi\frac{\p
\ub_{B,1}}{\p\phi}+\ub_{B,1}-\bub_{B,1}&=&\frac{1}{1-\e\eta}\cos\phi\frac{\p
\ub_{B,0}}{\p\theta},\label{expansion temp 6}\\
\sin\phi\frac{\p \ub_{B,2}}{\p\eta}-\frac{\e}{1-\e\eta}\cos\phi\frac{\p
\ub_{B,2}}{\p\phi}+\ub_{B,2}-\bub_{B,2}&=&\frac{1}{1-\e\eta}\cos\phi\frac{\p
\ub_{B,1}}{\p\theta}-\frac{\p\ub_{B,0}}{\p t},\label{expansion temp 7}\\
\ldots\nonumber\\
\sin\phi\frac{\p \ub_{B,k}}{\p\eta}-\frac{\e}{1-\e\eta}\cos\phi\frac{\p
\ub_{B,k}}{\p\phi}+\ub_{B,k}-\bub_{B,k}&=&\frac{1}{1-\e\eta}\cos\phi\frac{\p
\ub_{B,k-1}}{\p\theta}-\frac{\p\ub_{B,k-2}}{\p t}.
\end{eqnarray}
where
\begin{eqnarray}
\bub_{B,k}(t,\eta,\theta)=\frac{1}{2\pi}\int_{-\pi}^{\pi}\ub_{B,k}(t,\eta,\theta,\phi)\ud{\phi}.
\end{eqnarray}
It is important to note the solution $\ub_{B,k}$ depends on $\e$ and this is the reason why we add the superscript $\e$ to $\u_k$, $\ub_{I,k}$ and $\ub_{B,k}$.

\subsection{Initial-Boundary Layer Expansion}

Above construction of initial layer and boundary layer yields an interesting fact that at the corner point $(t,\vx)=(0,\vx_0)$ for $\vx_0\in\p\Omega$, the initial layer starting from this point has a contribution on the boundary data, and the boundary layer starting from this point has a contribution on the initial data. Therefore, we have to find some additional functions to compensate these noises. The classical theory of asymptotic analysis requires the so-called initial-boundary layer, where both the temporal scaling and spacial scaling should be used simultaneously. Fortunately, based on our analysis, the improved compatibility condition (\ref{improved compatibility condition}) implies the value at this corner point is a constant for $\vw\cdot\vec n<0$. Then
These contribution must be zero at the zeroth order, i.e.
\begin{eqnarray}
\ub_{I,0}(\tau,\vx_0,\vw)&=&0,\\
\ub_{B,0}(0,\eta,\theta,\phi)&=&0,
\end{eqnarray}
Therefore, the zeroth order initial-boundary layer is absent.

\subsection{Construction of Asymptotic Expansion}

The bridge between the interior solution, the initial layer, and the boundary layer
is the initial and boundary condition of (\ref{transport}). To avoid the introduction of higher order initial-boundary layer, we only require the zeroth order expansion of initial and boundary data be satisfied, i.e. we have
\begin{eqnarray}
\u_0(0,\vx,\vw)+\ub_{I,0}(0,\vx,\vw)+\ub_{B,0}(0,\vx,\vw)&=&h(\vx,\vw),\\
\u_0(t,\vx_0,\vw)+\ub_{I,0}(t,\vx_0,\vw)+\ub_{B,0}(t,\vx_0,\vw)&=&g(t,\vx_0,\vw).
\end{eqnarray}
The construction of $\u_k$. $\ub_{I,k}$ and $\ub_{B,k}$ are as follows:\\
\ \\
Assume the cut-off function $\psi$ and $\psi_0$ are defined as
\begin{eqnarray}\label{cut-off 1}
\psi(\mu)=\left\{
\begin{array}{ll}
1&0\leq\mu\leq1/2,\\
0&3/4\leq\mu\leq\infty.
\end{array}
\right.
\end{eqnarray}
\begin{eqnarray}\label{cut-off 2}
\psi_0(\mu)=\left\{
\begin{array}{ll}
1&0\leq\mu\leq1/4,\\
0&3/8\leq\mu\leq\infty.
\end{array}
\right.
\end{eqnarray}
and define the force as
\begin{eqnarray}\label{force}
F(\e;\eta)=-\frac{\e\psi(\e\eta)}{1-\e\eta},
\end{eqnarray}
\ \\
Step 1: Construction of zeroth order terms.\\
The zeroth order boundary layer solution is defined as
\begin{eqnarray}\label{expansion temp 9}
\left\{
\begin{array}{rcl}
\ub_{B,0}(t,\eta,\theta,\phi)&=&\psi_0(\e\eta)\bigg(\f_0^{\e}(t,\eta,\theta,\phi)-f_0^{\e}(t,\infty,\theta)\bigg),\\
\sin\phi\dfrac{\p \f_0^{\e}}{\p\eta}-F(\e;\eta)\cos\phi\dfrac{\p
\f_0^{\e}}{\p\phi}+\f_0^{\e}-\bar \f_0^{\e}&=&0,\\
\f_0^{\e}(t,0,\theta,\phi)&=&g(t,\theta,\phi)\ \ \text{for}\ \
\sin\phi>0,\\\rule{0ex}{1em}
\lim_{\eta\rt\infty}\f_0^{\e}(t,\eta,\theta,\phi)&=&f_0^{\e}(t,\infty,\theta).
\end{array}
\right.
\end{eqnarray}
Assuming
$g\in L^{\infty}$, by Theorem \ref{Milne theorem 1}, we can show
there exists a unique solution $\f_0^{\e}(t,\eta,\theta,\phi)\in
L^{\infty}$. Hence, $\ub_{B,0}$ is well-defined.\\
The zeroth order initial layer is defined as
\begin{eqnarray}\label{expansion temp 21}
\left\{
\begin{array}{rcl}
\ub_{I,0}(\tau,\vx,\vw)&=&\ff_0^{\e}(\tau,\vx,\vw)-\ff_0^{\e}(\infty,\vx)\\
\p_{\tau}\bar\ff_0^{\e}&=&0,\\\rule{0ex}{1.0em}
\ff_0^{\e}(\tau,\vx,\vw)&=&\ue^{-\tau}\ff_0^{\e}(0,\vx,\vw)+(1-\ue^{-\tau})\bar\ff_0^{\e}(0,\vx),\\
\ff_0^{\e}(0,\vx,\vw)&=&h(\vx,\vw),\\
\lim_{\tau\rt\infty}\ff_0^{\e}(\tau,\vx,\vw)&=&\ff_0^{\e}(\infty,\vx).
\end{array}
\right.
\end{eqnarray}
Assuming
$h\in L^{\infty}$. Then we can show
there exists a unique solution $\ff_0^{\e}(\tau,\vx,\vw)\in
L^{\infty}$. Hence, $\ub_{I,0}$ is well-defined.\\
Then we can define the zeroth order interior solution as
\begin{eqnarray}\label{expansion temp 8}
\left\{
\begin{array}{rcl}
\u_0&=&\bu_0,\\\rule{0ex}{1em}
\dt\bu_0-\Delta_x\bu_0&=&0,\\\rule{0ex}{1em}\bu_0(0,\vx)&=&\ff_0^{\e}(\infty,\vx)\
\ \text{in}\ \ \Omega,\\\rule{0ex}{1em}
\bu_0(t,\vx_0)&=&f_0^{\e}(t,\infty,\theta)\ \ \text{on}\ \ \p\Omega,
\end{array}
\right.
\end{eqnarray}
where $(t,\vx,\vw)$ is the same point as $(\tau,\eta,\theta,\phi)$. Note that due to the improved compatibility condition (\ref{improved compatibility condition}), we have
$\ub_{B,0}(0,\eta,\theta,\phi)=\ub_{I,0}(\tau,\vx_0,\vw)=0$.\\
\ \\
Step 2: Construction of first order terms. \\
Define the first order boundary layer solution as
\begin{eqnarray}\label{expansion temp 11}
\left\{
\begin{array}{rcl}
\ub_{B,1}(t,\eta,\theta,\phi)&=&\psi_0(\e\eta)\bigg(\f_1^{\e}(t,\eta,\theta,\phi)-f_1^{\e}(t,\infty,\theta)\bigg),\\
\sin\phi\dfrac{\p \f_1^{\e}}{\p\eta}-F(\e;\eta)\cos\phi\dfrac{\p
\f_1^{\e}}{\p\phi}+\f_1^{\e}-\bar
\f_1^{\e}&=&\cos\phi\dfrac{\psi(\e\eta)}{1-\e\eta}\dfrac{\p
\ub_{B,0}}{\p\theta},\\\rule{0ex}{1em}
\f_1^{\e}(t,0,\theta,\phi)&=&\vw\cdot\nx\u_0(t,\vx_0,\vw)\ \ \text{for}\
\ \sin\phi>0,\\\rule{0ex}{1em}
\lim_{\eta\rt\infty}\f_1^{\e}(t,\eta,\theta,\phi)&=&f_1^{\e}(t,\infty,\theta).
\end{array}
\right.
\end{eqnarray}
Define the first order initial layer as
\begin{eqnarray}\label{expansion temp 22}
\left\{
\begin{array}{rcl}
\ub_{I,1}(\tau,\vx,\vw)&=&\ff_1^{\e}(\tau,\vx,\vw)-\ff_1^{\e}(\infty,\vx)\\
\p_{\tau}\bar\ff_1^{\e}&=&-\displaystyle\int_{\s^1}\bigg(\vw\cdot\nabla_x\ub_{I,0}\bigg)\ud{\vw},\\\rule{0ex}{1.5em}
\ff_1^{\e}(\tau,\vx,\vw)&=&\ue^{-\tau}\ff_1^{\e}(0,\vx,\vw)+\displaystyle\int_0^{\tau}\bigg(\bar\ff_1^{\e}-\vw\cdot\nabla_x\ub_{I,0}\bigg)(s,\vx,\vw)\ue^{s-\tau}\ud{s},\\
\ff_1^{\e}(0,\vx,\vw)&=&\vw\cdot\nx\u_0(0,\vx,\vw),\\
\lim_{\tau\rt\infty}\ff_1^{\e}(\tau,\vx,\vw)&=&\ff_1^{\e}(\infty,\vx).
\end{array}
\right.
\end{eqnarray}
Define the first order interior solution as
\begin{eqnarray}\label{expansion temp 10}
\left\{
\begin{array}{rcl}
\u_1&=&\bu_1-\vw\cdot\nx\u_0,\\
\dt\bu_1-\Delta_x\bu_1&=&0,\\\rule{0ex}{1em}\bu_1(0,\vx)&=&\ff_1^{\e}(\infty,\vx)\ \
\text{in}\ \ \Omega,\\\rule{0ex}{1em}
\bu_1(t,\vx)&=&f_1^{\e}(t,\infty,\theta)\ \ \text{on}\ \ \p\Omega.
\end{array}
\right.
\end{eqnarray}
\ \\
Step 3: Construction of $\ub_2$ and $\u_2$. \\
Define the second order boundary layer solution as
\begin{eqnarray}
\
\end{eqnarray}
\begin{eqnarray}
\left\{
\begin{array}{rcl}
\ub_{B,2}(t,\eta,\theta,\phi)&=&\psi_0(\e\eta)\bigg(\f_2^{\e}(t,\eta,\theta,\phi)-f_2^{\e}(t,\infty,\theta)\bigg),\\
\sin\phi\dfrac{\p \f_2^{\e}}{\p\eta}-F(\e;\eta)\cos\phi\dfrac{\p
\f_2^{\e}}{\p\phi}+\f_2^{\e}-\bar
\f_2^{\e}&=&\cos\phi\dfrac{\psi(\e\eta)}{1-\e\eta}\dfrac{\p
\ub_{B,1}}{\p\theta}-\dfrac{\p\ub_{B,0}}{\p t},\\\rule{0ex}{1em}
\f_2^{\e}(t,0,\theta,\phi)&=&\vw\cdot\nx\u_1(t,\vx_0,\vw)+\dt\u_0(t,\vx_0,\vw)\
\ \text{for}\ \ \sin\phi>0,\\\rule{0ex}{1em}
\lim_{\eta\rt\infty}\f_2^{\e}(t,\eta,\theta,\phi)&=&f_2^{\e}(t,\infty,\theta).\no
\end{array}
\right.
\end{eqnarray}
Define the second order initial layer as
\begin{eqnarray}\label{expansion temp 23}
\left\{
\begin{array}{rcl}
\ub_{I,2}(\tau,\vx,\vw)&=&\ff_2^{\e}(\tau,\vx,\vw)-\ff_2^{\e}(\infty,\vx)\\
\p_{\tau}\bar\ff_2^{\e}&=&-\displaystyle\int_{\s^1}\bigg(\vw\cdot\nabla_x\ub_{I,1}\bigg)\ud{\vw},\\\rule{0ex}{1.5em}
\ff_2^{\e}(\tau,\vx,\vw)&=&\ue^{-\tau}\ff_2^{\e}(0,\vx,\vw)+\displaystyle\int_0^{\tau}\bigg(\bar\ff_2^{\e}-\vw\cdot\nabla_x\ub_{I,1}\bigg)(s,\vx,\vw)\ue^{s-\tau}\ud{s},\\
\ff_2^{\e}(0,\vx,\vw)&=&\vw\cdot\nx\u_1(0,\vx,\vw)+\dt\u_0(0,\vx,\vw),\\
\lim_{\tau\rt\infty}\ff_2^{\e}(\tau,\vx,\vw)&=&\ff_2^{\e}(\infty,\vx).
\end{array}
\right.
\end{eqnarray}
Define the first order interior solution as
\begin{eqnarray}
\left\{
\begin{array}{rcl}
\u_2&=&\bu_2-\vw\cdot\nx\u_1-\dt\u_0,\\
\dt\bu_2-\Delta_x\bu_2&=&0,\\\rule{0ex}{1em}\bu_2(0,\vx)&=&\ff_2^{\e}(\infty,\vx)\ \
\text{in}\ \ \Omega,\\\rule{0ex}{1em}
\bu_2(t,\vx)&=&f_2^{\e}(t,\infty,\theta)\ \ \text{on}\ \ \p\Omega.
\end{array}
\right.
\end{eqnarray}
\ \\
Step 4: Generalization to arbitrary $k$.\\
Similar to above procedure, we can define the $k^{th}$ order
boundary layer solution as
\begin{eqnarray}
\
\end{eqnarray}
\begin{eqnarray}
\left\{
\begin{array}{rcl}
\ub_{B,k}(t,\eta,\theta,\phi)&=&\psi_0(\e\eta)\bigg(\f_k^{\e}(t,\eta,\theta,\phi)-f_k^{\e}(t,\infty,\theta)\bigg),\\
\sin\phi\dfrac{\p \f_k^{\e}}{\p\eta}-F(\e;\eta)\cos\phi\dfrac{\p
\f_k^{\e}}{\p\phi}+\f_k^{\e}-\bar
\f_k^{\e}&=&\cos\phi\dfrac{\psi(\e\eta)}{1-\e\eta}\dfrac{\p
\ub_{B,k-1}}{\p\theta}-\dfrac{\p\ub_{B,k-2}}{\p t},\\\rule{0ex}{1em}
\f_k^{\e}(t,0,\theta,\phi)&=&\vw\cdot\nx\u_{k-1}(t,\vx_0,\vw)+\dt\u_{k-2}(t,\vx_0,\vw)\
\ \text{for}\ \ \sin\phi>0,\\\rule{0ex}{1em}
\lim_{\eta\rt\infty}\f_k^{\e}(t,\eta,\theta,\phi)&=&f_k^{\e}(t,\infty,\theta).\no
\end{array}
\right.
\end{eqnarray}
Define the $k^{th}$ order initial layer as
\begin{eqnarray}\label{expansion temp 24}
\left\{
\begin{array}{rcl}
\ub_{I,k}(\tau,\vx,\vw)&=&\ff_k^{\e}(\tau,\vx,\vw)-\ff_k^{\e}(\infty,\vx)\\
\p_{\tau}\bar\ff_k^{\e}&=&-\displaystyle\int_{\s^1}\bigg(\vw\cdot\nabla_x\ub_{I,k-1}\bigg)\ud{\vw},\\\rule{0ex}{1.5em}
\ff_k^{\e}(\tau,\vx,\vw)&=&\ue^{-\tau}\ff_k^{\e}(0,\vx,\vw)+\displaystyle\int_0^{\tau}\bigg(\bar\ff_k^{\e}-\vw\cdot\nabla_x\ub_{I,k-1}\bigg)(s,\vx,\vw)\ue^{s-\tau}\ud{s},\\
\ff_k^{\e}(0,\vx,\vw)&=&\vw\cdot\nx\u_{k-1}(0,\vx,\vw)+\dt\u_{k-2}(0,\vx,\vw),\\
\lim_{\tau\rt\infty}\ff_k^{\e}(\tau,\vx,\vw)&=&\ff_k^{\e}(\infty,\vx).
\end{array}
\right.
\end{eqnarray}
Define the $k^{th}$ order interior solution as
\begin{eqnarray}
\left\{
\begin{array}{rcl}
\u_k&=&\bu_k-\vw\cdot\nx\u_{k-1}-\dt\u_{k-2},\\\rule{0ex}{1em}
\dt\bu_k-\Delta_x\bu_k&=&0,\\\rule{0ex}{1em}\bu_k(0,\vx)&=&\ff_k^{\e}(\infty,\vx)\ \
\text{in}\ \ \Omega,\\\rule{0ex}{1em} \bu_k&=&f_k^{\e}(t,\infty,\theta)\
\ \text{on}\ \ \p\Omega.
\end{array}
\right.
\end{eqnarray}
When $g$ and $h$ are sufficiently smooth, then all the functions defined above are well-posed. The key point here is in the boundary layer, the source term including $\p_{\theta}\ub_{B,k}$ is in $L^{\infty}$ due to the substitution (\ref{substitution 4}).

\section{$\e$-Milne Problem}

In this section, we study the $\e$-Milne problem for
$f^{\e}(\eta,\theta,\phi)$ in the domain
$(\eta,\theta,\phi)\in[0,\infty)\times[-\pi,\pi)\times[-\pi,\pi)$
\begin{eqnarray}\label{Milne problem}
\left\{
\begin{array}{rcl}\displaystyle
\sin\phi\frac{\p f^{\e}}{\p\eta}+F(\e;\eta)\cos\phi\frac{\p
f^{\e}}{\p\phi}+f^{\e}-\bar f^{\e}&=&S^{\e}(\eta,\theta,\phi),\\
f^{\e}(0,\theta,\phi)&=&H^{\e}(\theta,\phi)\ \ \text{for}\ \ \sin\phi>0,\\\rule{0ex}{1.0em}
\lim_{\eta\rt\infty}f^{\e}(\eta,\theta,\phi)&=&f^{\e}_{\infty}(\theta),
\end{array}
\right.
\end{eqnarray}
where
\begin{eqnarray}\label{Milne average}
\bar
f^{\e}(\eta,\theta)=\frac{1}{2\pi}\int_{-\pi}^{\pi}f^{\e}(\eta,\theta,\phi)\ud{\phi},
\end{eqnarray}
\begin{eqnarray}
F(\e;\eta)=-\frac{\e\psi(\e\eta)}{1-\e\eta},
\end{eqnarray}
\begin{eqnarray}
\psi(\mu)=\left\{
\begin{array}{ll}
1&0\leq\mu\leq1/2,\\
0&3/4\leq\mu\leq\infty,
\end{array}
\right.
\end{eqnarray}
\begin{eqnarray}\label{Milne bounded}
\abs{H^{\e}(\theta,\phi)}\leq M,
\end{eqnarray}
and
\begin{eqnarray}\label{Milne decay}
\abs{S^{\e}(\eta,\theta,\phi)}\leq Me^{-K\eta},
\end{eqnarray}
for $M>0$ and $K>0$ uniform in $\e$ and $\theta$. In this section, since the key variables here are $\eta$ and $\phi$, we temporarily ignore the dependence on $\e$ and $\theta$. We define the norms in the
space $(\eta,\phi)\in[0,\infty)\times[-\pi,\pi)$ as follows:
\begin{eqnarray}
\tnnm{f}&=&\bigg(\int_0^{\infty}\int_{-\pi}^{\pi}\abs{f(\eta,\phi)}^2\ud{\phi}\ud{\eta}\bigg)^{1/2},\\
\lnnm{f}&=&\sup_{(\eta,\phi)\in[0,\infty)\times[-\pi,\pi)}\abs{f(\eta,\phi)}.
\end{eqnarray}
In \cite[Section 4]{AA003}, the authors proved the following results:
\begin{theorem}\label{Milne theorem 1}
There exists a unique solution $f(\eta,\phi)$ to the $\e$-Milne
problem (\ref{Milne problem}) satisfying
\begin{eqnarray}
\tnnm{f-f_{\infty}}\leq C\bigg(1+M+\frac{M}{K}\bigg).
\end{eqnarray}
\end{theorem}
\begin{theorem}\label{Milne theorem 2}
There exists a unique solution $f(\eta,\phi)$ to the $\e$-Milne
problem (\ref{Milne problem}) satisfying
\begin{eqnarray}
\lnnm{f-f_{\infty}}\leq C\bigg(1+M+\frac{M}{K}\bigg).
\end{eqnarray}
\end{theorem}
\begin{theorem}\label{Milne theorem 3}
For
$K_0>0$ sufficiently small, the solution $f(\eta,\phi)$ to the
$\e$-Milne problem (\ref{Milne problem}) satisfies
\begin{eqnarray}
\tnnm{e^{K_0\eta}(f-f_{\infty})}\leq C\bigg(1+M+\frac{M}{K}\bigg),
\end{eqnarray}
\end{theorem}
\begin{theorem}\label{Milne theorem 4}
For
$K_0>0$ sufficiently small, the solution $f(\eta,\phi)$ to the
$\e$-Milne problem (\ref{Milne problem}) satisfies
\begin{eqnarray}
\lnnm{e^{K_0\eta}(f-f_{\infty})}\leq C\bigg(1+M+\frac{M}{K}\bigg),
\end{eqnarray}
\end{theorem}
\begin{theorem}\label{Milne theorem 5}
The solution $f(\eta,\phi)$ to the $\e$-Milne problem (\ref{Milne
problem}) with $S=0$ satisfies the maximum principle, i.e.
\begin{eqnarray}
\min_{\sin\phi>0}h(\phi)\leq f(\eta,\phi)\leq
\max_{\sin\phi>0}h(\phi).
\end{eqnarray}
\end{theorem}
\begin{remark}\label{Milne remark}
Note that when $F=0$, Theorem \ref{Milne theorem 1}, Theorem \ref{Milne theorem 2}, Theorem \ref{Milne theorem 3}, Theorem \ref{Milne theorem 4}, and
Theorem \ref{Milne theorem 5} still hold. Hence, we can deduce the
well-posedness, decay and maximum principle of the classical Milne
problem
\begin{eqnarray}\label{classical Milne problem}
\left\{
\begin{array}{rcl}\displaystyle
\sin\phi\frac{\p f}{\p\eta}+f-\bar f&=&S(\eta,\phi),\\
f(0,\phi)&=&h(\phi)\ \ \text{for}\ \ \sin\phi>0,\\
\lim_{\eta\rt\infty}f(\eta,\phi)&=&f_{\infty}.
\end{array}
\right.
\end{eqnarray}
\end{remark}
%
%

\section{Diffusive Limit}

In this section, we prove the first part of Theorem \ref{main 2}.
\begin{theorem}
Assume $g(t,\vx_0,\vw)\in C^4([0,\infty)\times\Gamma^-)$ and
$h(\vx,\vw)\in C^4(\Omega\times\s^1)$. Then for the unsteady neutron transport
equation (\ref{transport}), the unique solution
$u^{\e}(t,\vx,\vw)\in L^{\infty}([0,\infty)\times\Omega\times\s^1)$
satisfies
\begin{eqnarray}\label{main theorem 2}
\lnm{u^{\e}-\u_0-\ub_{I,0}-\ub_{B,0}}=o(1),
\end{eqnarray}
where the interior solution $\u_0$ is
defined in (\ref{expansion temp 8}), the initial layer $\ub_{I,0}$ is defined in (\ref{expansion temp 21}), and the boundary layer $\ub_{B,0}$ is defined in (\ref{expansion temp 9}).
\end{theorem}
\begin{proof}
We divide the proof into several steps:\\
\ \\
Step 1: Remainder definitions.\\
We may rewrite the asymptotic expansion as follows:
\begin{eqnarray}
u^{\e}&\sim&\sum_{k=0}^{\infty}\e^k\u_k+\sum_{k=0}^{\infty}\e^k\ub_{I,k}+\sum_{k=0}^{\infty}\e^k\ub_{B,k}.
\end{eqnarray}
The remainder can be defined as
\begin{eqnarray}\label{pf 1}
R_N&=&u^{\e}-\sum_{k=0}^{N}\e^k\u_k-\sum_{k=0}^{N}\e^k\ub_{I,k}-\sum_{k=0}^{N}\e^k\ub_{B,k}=u^{\e}-\q_N-\qb_{I,N}-\qb_{B,N},
\end{eqnarray}
where
\begin{eqnarray}
\q_N&=&\sum_{k=0}^{N}\e^k\u_k,\\
\qb_{I,N}&=&\sum_{k=0}^{N}\e^k\ub_{I,k},\\
\qb_{B,N}&=&\sum_{k=0}^{N}\e^k\ub_{B,k}.
\end{eqnarray}
Noting the equation is equivalent to the
equations (\ref{initial temp}) and (\ref{transport temp}), we write $\ll$ to denote the neutron
transport operator as follows:
\begin{eqnarray}
\ll u&=&\e^2\dt u+\e\vw\cdot\nx u+u-\bar u\\
&=&\p_{\tau}u+\e\vw\cdot\nabla_xu+u-\bar u\\
&=&\e^2\frac{\p u}{\p t}+\sin\phi\frac{\p
u}{\p\eta}-\frac{\e}{1-\e\eta}\cos\phi\bigg(\frac{\p
u}{\p\phi}+\frac{\p u}{\p\theta}\bigg)+u-\bar u.\nonumber
\end{eqnarray}
\ \\
Step 2: Estimates of $\ll \q_N$.\\
The interior contribution can be estimated as
\begin{eqnarray}
\ll\q_0&=&\e^2\dt\q_0+\e\vw\cdot\nx \q_0+\q_0-\bar
\q_0\\
&=&\e^2\dt\u_0+ \e\vw\cdot\nx
\u_0+(\u_0-\bu_0)=\e^2\dt\u_0+\e\vw\cdot\nx \u_0.\no
\end{eqnarray}
We have
\begin{eqnarray}
\abs{\e^2\dt\u_0}&\leq&C\e^2\abs{\dt\u_0}\leq C\e^2,\\
\abs{\e\vw\cdot\nx \u_0}&\leq& C\e\abs{\nx \u_0}\leq C\e.
\end{eqnarray}
This implies
\begin{eqnarray}
\abs{\ll \q_0}\leq C\e.
\end{eqnarray}
Similarly, for higher order term, we can estimate
\begin{eqnarray}
\ll\q_N=\e^2\dt\q_N+\e\vw\cdot\nx \q_N+\q_N-\bar
\q_N&=&\e^{N+2}\dt\u_N+\e^{N+1}\vw\cdot\nx \u_N.
\end{eqnarray}
We have
\begin{eqnarray}
\abs{\e^{N+1}\dt\u_N}&\leq&C\e^{N+2}\abs{\dt\u_N}\leq C\e^{N+2},\\
\abs{\e^{N+1}\vw\cdot\nx \u_N}&\leq& C\e^{N+1}\abs{\nx \u_N}\leq
C\e^{N+1}.
\end{eqnarray}
This implies
\begin{eqnarray}\label{pf 2}
\abs{\ll \q_N}\leq C\e^{N+1}.
\end{eqnarray}
\ \\
Step 3: Estimates of $\ll \qb_{I,N}$.\\
The initial layer contribution can be estimated as
\begin{eqnarray}
\ll\qb_{I,0}&=&\p_{\tau}\qb_{I,0}+\e\vw\cdot\nabla_x\qb_{I,0}+\qb_{I,0}-\bar \qb_{I,0}\\
&=&\p_{\tau}\ub_{I,0}+\e\vw\cdot\nabla_x\ub_{I,0}+\ub_{I,0}-\bar \ub_{I,0}=\e\nabla_x\ub_{I,0}.\no
\end{eqnarray}
Based on the smoothness of $\ub_{I,0}$, we have
\begin{eqnarray}
\abs{\ll\qb_{I,0}}=\abs{\e\nabla_x\ub_{I,0}}\leq C\e.
\end{eqnarray}
Similarly, we have
\begin{eqnarray}
\ll\qb_{I,N}&=&\p_{\tau}\qb_{I,N}+\e\vw\cdot\nabla_x\qb_{I,N}+\qb_{I,N}-\bar \qb_{I,N}=\e^{N+1}\nabla_x\ub_{I,N}.
\end{eqnarray}
Therefore, we have
\begin{eqnarray}\label{pf 4}
\abs{\ll\qb_{I,N}}=\abs{\e^{N+1}\nabla_x\ub_{I,N}}\leq C\e^{N+1}.
\end{eqnarray}
\ \\
Step 4: Estimates of $\ll \qb_{B,N}$.\\
The boundary layer solution is
$\ub_k=(f_k^{\e}-f_k^{\e}(\infty))\cdot\psi_0=\v_k\psi_0$ where
$f_k^{\e}(\eta,\theta,\phi)$ solves the $\e$-Milne problem and
$\v_k=f_k^{\e}-f_k^{\e}(\infty)$. Notice $\psi_0\psi=\psi_0$, so the
boundary layer contribution can be estimated as
\begin{eqnarray}\label{remainder temp 1}
\\
\ll\qb_{B,0}&=&\e^2\frac{\p \qb_{B,0}}{\p t}+\sin\phi\frac{\p
\qb_{B,0}}{\p\eta}-\frac{\e}{1-\e\eta}\cos\phi\bigg(\frac{\p
\qb_{B,0}}{\p\phi}+\frac{\p \qb_{B,0}}{\p\theta}\bigg)+\qb_{B,0}-\bar
\qb_{B,0}\no\\
&=&\e^2\frac{\p \v_0}{\p t}+\sin\phi\bigg(\psi_0\frac{\p
\v_0}{\p\eta}+\v_0\frac{\p\psi_0}{\p\eta}\bigg)-\frac{\psi_0\e}{1-\e\eta}\cos\phi\bigg(\frac{\p
\v_0}{\p\phi}+\frac{\p \v_0}{\p\theta}\bigg)+\psi_0 \v_0-\psi_0\bar\v_0\nonumber\\
&=&\e^2\frac{\p \v_0}{\p t}+\sin\phi\bigg(\psi_0\frac{\p
\v_0}{\p\eta}+\v_0\frac{\p\psi_0}{\p\eta}\bigg)-\frac{\psi_0\psi\e}{1-\e\eta}\cos\phi\bigg(\frac{\p
\v_0}{\p\phi}+\frac{\p \v_0}{\p\theta}\bigg)+\psi_0 \v_0-\psi_0\bar\v_0\nonumber\\
&=&\e^2\frac{\p \v_0}{\p t}+\psi_0\bigg(\sin\phi\frac{\p
\v_0}{\p\eta}-\frac{\e\psi}{1-\e\eta}\cos\phi\frac{\p
\v_0}{\p\phi}+\v_0-\bar\v_0\bigg)+\sin\phi
\frac{\p\psi_0}{\p\eta}\v_0-\frac{\psi_0\e}{1-\e\eta}\cos\phi\frac{\p
\v_0}{\p\theta}\nonumber\\
&=&\e^2\frac{\p \v_0}{\p t}+\sin\phi
\frac{\p\psi_0}{\p\eta}\v_0-\frac{\psi_0\e}{1-\e\eta}\cos\phi\frac{\p
\v_0}{\p\theta}\nonumber.
\end{eqnarray}
It is easy to see
\begin{eqnarray}
\abs{\e^2\frac{\p \v_0}{\p t}}\leq \e^2\abs{\frac{\p \v_0}{\p
t}}\leq C\e^2.
\end{eqnarray}
Since $\psi_0=1$ when $\eta\leq 1/(4\e)$, the effective region of
$\px\psi_0$ is $\eta\geq1/(4\e)$ which is further and further from
the origin as $\e\rt0$. By Theorem \ref{Milne theorem 2}, the first
term in (\ref{remainder temp 1}) can be controlled as
\begin{eqnarray}
\abs{\sin\phi\frac{\p\psi_0}{\p\eta}\v_0}&\leq&
C\ue^{-\frac{K_0}{\e}}\leq C\e.
\end{eqnarray}
For the second term in (\ref{remainder temp 1}), we have
\begin{eqnarray}
\abs{-\frac{\psi_0\e}{1-\e\eta}\cos\phi\frac{\p
\v_0}{\p\theta}}&\leq&C\e\abs{\frac{\p \v_0}{\p\theta}}\leq C\e.
\end{eqnarray}
This implies
\begin{eqnarray}
\abs{\ll \qb_{B,0}}\leq C\e.
\end{eqnarray}
Similarly, for higher order term, we can estimate
\begin{eqnarray}\label{remainder temp 2}
\ll\qb_{B,N}&=&\e^2\frac{\p \qb_{B,N}}{\p t}+\sin\phi\frac{\p
\qb_{B,N}}{\p\eta}-\frac{\e}{1-\e\eta}\cos\phi\bigg(\frac{\p
\qb_{B,N}}{\p\phi}+\frac{\p \qb_{B,N}}{\p\theta}\bigg)+\qb_{B,N}-\bar
\qb_{B,N}\\
&=&\e^{N+2}\frac{\p \v_N}{\p t}+\sum_{i=0}^k\e^i\sin\phi
\frac{\p\psi_0}{\p\eta}\v_i-\frac{\psi_0\e^{k+1}}{1-\e\eta}\cos\phi\frac{\p
\v_k}{\p\theta}\nonumber.
\end{eqnarray}
It is obvious that
\begin{eqnarray}
\abs{\e^{N+2}\frac{\p \v_N}{\p t}}\leq \e^{N+2}\abs{\frac{\p
\v_N}{\p t}}\leq C\e^{N+2}.
\end{eqnarray}
Away from the origin, the first term in (\ref{remainder temp 2}) can
be controlled as
\begin{eqnarray}
\abs{\sum_{i=0}^k\e^i\sin\phi \frac{\p\psi_0}{\p\eta}\v_i}&\leq&
C\ue^{-\frac{K_0}{\e}}\leq C\e^{k+1}.
\end{eqnarray}
For the second term in (\ref{remainder temp 2}), we have
\begin{eqnarray}
\abs{-\frac{\psi_0\e^{k+1}}{1-\e\eta}\cos\phi\frac{\p
\v_k}{\p\theta}}&\leq&C\e^{k+1}\abs{\frac{\p \v_k}{\p\theta}}\leq
C\e^{k+1}.
\end{eqnarray}
This implies
\begin{eqnarray}\label{pf 3}
\abs{\ll \qb_{B,N}}\leq C\e^{k+1}.
\end{eqnarray}
\ \\
Step 5: Synthesis.\\
In summary, since $\ll u^{\e}=0$, collecting (\ref{pf 1}), (\ref{pf
2}), (\ref{pf 4}), and (\ref{pf 3}), we can prove
\begin{eqnarray}
\abs{\ll R_N}\leq C\e^{N+1}.
\end{eqnarray}
Consider the asymptotic expansion to $N=4$, then the remainder $R_4$
satisfies the equation
\begin{eqnarray}
\\
\left\{
\begin{array}{rcl}
\e\dt R_4+\e \vw\cdot\nabla_x R_4+R_4-\bar R_4&=&\ll R_4,\\
R_4(0,\vx,\vw)&=&\sum_{k=1}^4\e^k\ub_{B,k}(0,\vx,\vw),\\
R_4(t,\vx_0,\vw)&=&\sum_{k=1}^4\e^k\ub_{I,k}(t,\vx_0,\vw)\ \ \text{for}\ \ \vw\cdot\vec n<0\ \
\text{and}\ \ \vx_0\in\p\Omega.\no
\end{array}
\right.
\end{eqnarray}
Note that the initial data and boundary data are nonzero due the contribution of initial layer and boundary data at the point $(t,\vx)=(0,\vx_0)$.
By Theorem \ref{improved LI estimate}, we have
\begin{eqnarray}
\im{R_4}{[0,\infty)\times\Omega\times\s^1}
&\leq& C(\Omega)\bigg(\frac{1}{\e^{4}}\tm{\ll
R_4}{[0,\infty)\times\Omega\times\s^1}+\im{\ll
R_4}{[0,\infty)\times\Omega\times\s^1}\bigg)\\
&&+\im{\sum_{k=1}^4\e^k\ub_{B,k}(0,\vx,\vw)}{\Omega\times\s^1}\no\\
&&+\im{\sum_{k=1}^4\e^k\ub_{I,k}(t,\vx_0,\vw)}{[0,t]\times\Gamma^-}\no\\
&\leq&C(\Omega)\bigg(\frac{1}{\e^{4}}(C\e^5)+(C\e^5)\bigg)+C\e+C\e=C(\Omega)\e.\no
\end{eqnarray}
Hence, we have
\begin{eqnarray}
\nm{u^{\e}-\sum_{k=0}^4\e^k\u_k-\sum_{k=0}^4\e^k\ub_{I,k}-\sum_{k=0}^4\e^k\ub_{B,k}}_{L^{\infty}([0,\infty)\times\Omega\times\s^1)}=o(1).
\end{eqnarray}
Since it is easy to see
\begin{eqnarray}
\nm{\sum_{k=1}^4\e^k\u_k+\sum_{k=1}^4\e^k\ub_{I,k}+\sum_{k=1}^4\e^k\ub_{B,k}}_{L^{\infty}(\Omega\times\s^1)}=O(\e),
\end{eqnarray}
our result naturally follows.
\end{proof}

\section{Counterexample for Classical Approach}

In this section, we present the classical approach in \cite{Bensoussan.Lions.Papanicolaou1979} to construct asymptotic expansion, especially the boundary layer expansion, and give a counterexample to show this method is problematic in unsteady equation.

\subsection{Discussion on Expansions except Boundary Layer}

Basically, the expansions for interior solution and initial layer are identical to our method, so omit the details and only present the notation.
We define the interior expansion as follows:
\begin{eqnarray}\label{interior expansion.}
\uc(t,\vx,\vw)\sim\sum_{k=0}^{\infty}\e^k\uc_k(t,\vx,\vw),
\end{eqnarray}
$\uc_0(t,\vx,\vw)$ satisfies the equation
\begin{eqnarray}\label{interior 1.}
\left\{
\begin{array}{rcl}
\uc_0&=&\buc_0,\\
\dt\buc_0-\Delta_x\buc_0&=&0.
\end{array}
\right.
\end{eqnarray}
$\uc_1(t,\vx,\vw)$ satisfies
\begin{eqnarray}\label{interior 2.}
\left\{
\begin{array}{rcl}
\uc_1&=&\buc_1-\vw\cdot\nx\uc_0,\\
\dt\buc_1-\Delta_x\buc_1&=&0,
\end{array}
\right.
\end{eqnarray}
and $\uc_k(t,\vx,\vw)$ for $k\geq2$ satisfies
\begin{eqnarray}\label{interior 3.}
\left\{
\begin{array}{rcl}
\uc_k&=&\buc_k-\vw\cdot\nx\uc_{k-1}-\dt\uc_{k-2},\\
\dt\buc_k-\Delta_x\buc_k&=&0.
\end{array}
\right.
\end{eqnarray}
With the substitution (\ref{substitution 0}),
we define the initial layer expansion as follows:
\begin{eqnarray}\label{initial layer expansion.}
\ubc_I(\tau,\vx,\vw)\sim\sum_{k=0}^{\infty}\e^k\ubc_{I,k}(\tau,\vx,\vw),
\end{eqnarray}
where $\ubc_{I,0}$ satisfies
\begin{eqnarray}
\left\{
\begin{array}{rcl}
\p_{\tau}\bubc_{I,0}&=&0,\\\rule{0ex}{1.0em}
\ubc_{I,0}(\tau,\vx,\vw)&=&\ue^{-\tau}\ubc_{I,0}(0,\vx,\vw)+(1-\ue^{-\tau})\bubc_{I,0}(0,\vx).
\end{array}
\right.
\end{eqnarray}
and $\ubc_{I,k}(\tau,\vx,\vw)$ for $k\geq1$ satisfies
\begin{eqnarray}
\left\{
\begin{array}{rcl}
\p_{\tau}\bubc_{I,k}&=&-\displaystyle\int_{\s^1}\bigg(\vw\cdot\nabla_x\ubc_{I,k-1}\bigg)\ud{\vw},\\\rule{0ex}{1.5em}
\ubc_{I,k}(\tau,\vx,\vw)&=&\ue^{-\tau}\ubc_{I,k}(0,\vx,\vw)+\displaystyle\int_0^{\tau}\bigg(\bubc_{I,k}-\vw\cdot\nabla_x\ubc_{I,k-1}\bigg)(s,\vx,\vw)\ue^{s-\tau}\ud{s}.
\end{array}
\right.
\end{eqnarray}

\subsection{Boundary Layer Expansion}

By the idea in \cite{Bensoussan.Lions.Papanicolaou1979}, the boundary layer expansion can be defined by introducing substitutions (\ref{substitution 1}), (\ref{substitution 2}), and (\ref{substitution 3}). Note that we terminate here and do not further use substitution (\ref{substitution 4}). Hence, we have the
transformed equation for (\ref{transport}) as
\begin{eqnarray}\label{classical temp.}
\left\{ \begin{array}{l}\displaystyle \e^2\frac{\p u^{\e}}{\p
t}+\sin(\theta+\xi)\frac{\p
u^{\e}}{\p\eta}-\frac{\e}{1-\e\eta}\cos(\theta+\xi)\frac{\p
u^{\e}}{\p\theta}+u^{\e}-\frac{1}{2\pi}\int_{-\pi}^{\pi}u^{\e}\ud{\xi}=0,\\\rule{0ex}{1.0em}
u^{\e}(0,\eta,\theta,\xi)=h(\eta,\theta),\\\rule{0ex}{1.0em}
u^{\e}(0,\theta,\xi)=g(\theta,\xi)\ \ \text{for}\ \
\sin(\theta+\xi)>0.
\end{array}
\right.
\end{eqnarray}
\ \\
We now define the Milne expansion of boundary layer as follows:
\begin{eqnarray}\label{classical expansion.}
\ubc(t,\eta,\theta,\phi)\sim\sum_{k=0}^{\infty}\e^k\ubc_k(t,\eta,\theta,\phi),
\end{eqnarray}
where $\ubc_k$ can be determined by comparing the order of $\e$ via
plugging (\ref{classical expansion.}) into the equation
(\ref{classical temp.}). Thus, in a neighborhood of the boundary, we
have
\begin{eqnarray}
\sin(\theta+\xi)\frac{\p
\ubc_0}{\p\eta}+\ubc_0-\bubc_0&=&0,\label{cexpansion temp 5.}\\
\sin(\theta+\xi)\frac{\p
\ubc_1}{\p\eta}+\ubc_1-\bubc_1&=&\frac{1}{1-\e\eta}\cos(\theta+\xi)\frac{\p
\ubc_0}{\p\theta},\label{cexpansion temp 6.}\\
\sin(\theta+\xi)\frac{\p
\ubc_2}{\p\eta}+\ubc_2-\bubc_2&=&\frac{1}{1-\e\eta}\cos(\theta+\xi)\frac{\p
\ubc_1}{\p\theta}-\frac{\p \ubc_0}{\p
t},\label{cexpansion temp 7.}\\
\ldots\nonumber\\
\sin(\theta+\xi)\frac{\p
\ubc_k}{\p\eta}+\ubc_k-\bubc_k&=&\frac{1}{1-\e\eta}\cos(\theta+\xi)\frac{\p
\ubc_{k-1}}{\p\theta}-\frac{\p \ubc_{k-2}}{\p t},
\end{eqnarray}
where
\begin{eqnarray}
\bar
\ubc_k(t,\eta,\theta)=\frac{1}{2\pi}\int_{-\pi}^{\pi}\ubc_k(t,\eta,\theta,\xi)\ud{\xi}.
\end{eqnarray}

\subsection{Classical Approach to Construct Asymptotic Expansion}

Similarly, we require the zeroth order expansion of initial and boundary data be satisfied, i.e. we have
\begin{eqnarray}
\uc_0(0,\vx,\vw)+\ubc_{I,0}(0,\vx,\vw)+\ubc_{B,0}(0,\vx,\vw)&=&h,\\
\uc_0(t,\vx_0,\vw)+\ubc_{I,0}(t,\vx_0,\vw)+\ubc_{B,0}(t,\vx_0,\vw)&=&g.
\end{eqnarray}
The construction of $\uc_k$, $\ubc_{I,k}$, and $\ubc_{B,k}$ by the idea in
\cite{Bensoussan.Lions.Papanicolaou1979} can be
summarized as follows:\\
\ \\
Assume the cut-off function $\psi$ and $\psi_0$ are defined as (\ref{cut-off 1}) and (\ref{cut-off 2}).\\
\ \\
Step 1: Construction of zeroth order terms.\\
The zeroth order boundary layer solution is defined as
\begin{eqnarray}\label{classical temp 1.}
\left\{
\begin{array}{rcl}
\ubc_0(t,\eta,\theta,\xi)&=&\psi_0(\e\eta)\bigg(\f_0(t,\eta,\theta,\xi)-f_0(t,\infty,\theta)\bigg),\\
\sin(\theta+\xi)\dfrac{\p \f_0}{\p\eta}+\f_0-\bar \f_0&=&0,\\
\f_0(t,0,\theta,\xi)&=&g(t,\theta,\xi)\ \ \text{for}\ \
\sin(\theta+\xi)>0,\\\rule{0ex}{1em}
\lim_{\eta\rt\infty}\f_0(t,\eta,\theta,\xi)&=&f_0(t,\infty,\theta).
\end{array}
\right.
\end{eqnarray}
The zeroth order initial layer is defined as
\begin{eqnarray}\label{classical temp 21.}
\left\{
\begin{array}{rcl}
\ubc_{I,0}(\tau,\vx,\vw)&=&\ff_0(\tau,\vx,\vw)-\ff_0(\infty,\vx)\\
\p_{\tau}\bar\ff_0&=&0,\\\rule{0ex}{1.0em}
\ff_0(\tau,\vx,\vw)&=&\ue^{-\tau}\ff_0(0,\vx,\vw)+(1-\ue^{-\tau})\bar\ff_0(0,\vx),\\
\ff_0(0,\vx,\vw)&=&h(\vx,\vw),\\
\lim_{\tau\rt\infty}\ff_0(\tau,\vx,\vw)&=&\ff_0(\infty,\vx).
\end{array}
\right.
\end{eqnarray}
Then we can define the
zeroth order interior solution as
\begin{eqnarray}\label{classical temp 2.}
\left\{
\begin{array}{rcl}
\uc_0&=&\buc_0,\\\rule{0ex}{1em}
\dt\buc_0-\Delta_x\buc_0&=&0,\\\rule{0ex}{1em}\buc_0(0,\vx)&=&\ff_0(\infty,\vx)\
\ \text{in}\ \ \Omega,\\\rule{0ex}{1em}
\buc_0(t,\vx_0)&=&\f_0(t,\infty,\theta)\ \ \text{on}\ \ \p\Omega,
\end{array}
\right.
\end{eqnarray}
where $(t,\vx,\vw)$ is the same point as $(\tau,\eta,\theta,\xi)$.\\
\ \\
Step 2: Construction of first order terms. \\
Define the first order boundary layer solution as
\begin{eqnarray}\label{classical temp 3.}
\left\{
\begin{array}{rcl}
\ubc_1(t,\eta,\theta,\xi)&=&\psi_0(\e\eta)\bigg(\f_1(t,\eta,\theta,\xi)-f_1(t,\infty,\theta)\bigg),\\
\sin(\theta+\xi)\dfrac{\p \f_1}{\p\eta}+\f_1-\bar
\f_1&=&\cos(\theta+\xi)\dfrac{\psi(\e\eta)}{1-\e\eta}\dfrac{\p
\ubc_0}{\p\theta},\\\rule{0ex}{1em}
\f_1(t,0,\theta,\xi)&=&\vw\cdot\nx\uc_0(t,\vx_0,\vw)\ \ \text{for}\
\ \sin(\theta+\xi)>0,\\\rule{0ex}{1em}
\lim_{\eta\rt\infty}\f_1(t,\eta,\theta,\xi)&=&f_1(t,\infty,\theta).
\end{array}
\right.
\end{eqnarray}
Define the first order initial layer as
\begin{eqnarray}\label{classical temp 22.}
\left\{
\begin{array}{rcl}
\ubc_{I,1}(\tau,\vx,\vw)&=&\ff_1(\tau,\vx,\vw)-\ff_1(\infty,\vx)\\
\p_{\tau}\bar\ff_1&=&-\displaystyle\int_{\s^1}\bigg(\vw\cdot\nabla_x\ubc_{I,0}\bigg)\ud{\vw},\\\rule{0ex}{1.5em}
\ff_1(\tau,\vx,\vw)&=&\ue^{-\tau}\ff_1(0,\vx,\vw)+\displaystyle\int_0^{\tau}\bigg(\bar\ff_1-\vw\cdot\nabla_x\ubc_{I,0}\bigg)(s,\vx,\vw)\ue^{s-\tau}\ud{s},\\
\ff_1(0,\vx,\vw)&=&\vw\cdot\nx\u_0(0,\vx,\vw),\\
\lim_{\tau\rt\infty}\ff_1(\tau,\vx,\vw)&=&\ff_1(\infty,\vx).
\end{array}
\right.
\end{eqnarray}
Define the first order interior solution as
\begin{eqnarray}\label{classical temp 5.}
\left\{
\begin{array}{rcl}
\uc_1&=&\buc_1-\vw\cdot\nx\uc_0,\\
\dt\buc_1-\Delta_x\buc_1&=&0,\\\rule{0ex}{1em}\buc_1(0,\vx)&=&\ff_1(\infty,\vx)\ \
\text{in}\ \ \Omega,\\\rule{0ex}{1em}
\buc_1(t,\vx)&=&f_1(t,\infty,\theta)\ \ \text{on}\ \ \p\Omega.
\end{array}
\right.
\end{eqnarray}
\ \\
Step 3: Construction of second order terms. \\
Define the second order boundary layer solution as
\begin{eqnarray}\label{classical temp 3.}
\left\{
\begin{array}{rcl}
\ubc_2(t,\eta,\theta,\xi)&=&\psi_0(\e\eta)\bigg(\f_2(t,\eta,\theta,\xi)-f_2(t,\infty,\theta)\bigg),\\
\sin(\theta+\xi)\dfrac{\p \f_2}{\p\eta}+\f_2-\bar
\f_2&=&\cos(\theta+\xi)\dfrac{\psi(\e\eta)}{1-\e\eta}\dfrac{\p
\ubc_1}{\p\theta}-\dfrac{\p\ubc_0}{\p t},\\\rule{0ex}{1em}
\f_2(t,0,\theta,\xi)&=&\vw\cdot\nx\uc_1(t,\vx_0,\vw)+\dt\uc_0(t,\vx_0,\vw)\
\ \text{for}\ \ \sin(\theta+\xi)>0,\\\rule{0ex}{1em}
\lim_{\eta\rt\infty}\f_2(t,\eta,\theta,\xi)&=&f_2(t,\infty,\theta).
\end{array}
\right.
\end{eqnarray}
Define the second order initial layer as
\begin{eqnarray}\label{classical temp 23.}
\left\{
\begin{array}{rcl}
\ubc_{I,2}(\tau,\vx,\vw)&=&\ff_2(\tau,\vx,\vw)-\ff_2(\infty,\vx)\\
\p_{\tau}\bar\ff_2&=&-\displaystyle\int_{\s^1}\bigg(\vw\cdot\nabla_x\ubc_{I,1}\bigg)\ud{\vw},\\\rule{0ex}{1.5em}
\ff_2(\tau,\vx,\vw)&=&\ue^{-\tau}\ff_2(0,\vx,\vw)+\displaystyle\int_0^{\tau}\bigg(\bar\ff_2-\vw\cdot\nabla_x\ubc_{I,1}\bigg)(s,\vx,\vw)\ue^{s-\tau}\ud{s},\\
\ff_2(0,\vx,\vw)&=&\vw\cdot\nx\u_1(0,\vx,\vw)+\dt\u_0(0,\vx,\vw),\\
\lim_{\tau\rt\infty}\ff_2(\tau,\vx,\vw)&=&\ff_2(\infty,\vx).
\end{array}
\right.
\end{eqnarray}
Define the first order interior solution as
\begin{eqnarray}\label{classical temp 5.}
\left\{
\begin{array}{rcl}
\uc_2&=&\buc_2-\vw\cdot\nx\uc_1-\dt\uc_0,\\
\dt\buc_2-\Delta_x\buc_2&=&0,\\\rule{0ex}{1em}\buc_2(0,\vx)&=&\ff_2(\infty,\vx)\ \
\text{in}\ \ \Omega,\\\rule{0ex}{1em}
\buc_2(t,\vx)&=&f_2(t,\infty,\theta)\ \ \text{on}\ \ \p\Omega.
\end{array}
\right.
\end{eqnarray}
\ \\
Step 4: Generalization to arbitrary $k$.\\
Similar to above procedure, we can define the $k^{th}$ order
boundary layer solution as
\begin{eqnarray}
\\
\left\{
\begin{array}{rcl}
\ubc_k(t,\eta,\theta,\xi)&=&\psi_0(\e\eta)\bigg(\f_k(t,\eta,\theta,\xi)-f_k(t,\infty,\theta)\bigg),\\
\sin(\theta+\xi)\dfrac{\p \f_k}{\p\eta}+\f_k-\bar
\f_k&=&\cos(\theta+\xi)\dfrac{\psi(\e\eta)}{1-\e\eta}\dfrac{\p
\ubc_{k-1}}{\p\theta}-\dfrac{\p\ubc_{k-2}}{\p t},\\\rule{0ex}{1em}
\f_k(t,0,\theta,\xi)&=&\vw\cdot\nx\uc_{k-1}(t,\vx_0,\vw)+\dt\uc_{k-2}(t,\vx_0,\vw)\
\ \text{for}\ \ \sin(\theta+\xi)>0,\\\rule{0ex}{1em}
\lim_{\eta\rt\infty}\f_k(t,\eta,\theta,\xi)&=&f_k(t,\infty,\theta).\no
\end{array}
\right.
\end{eqnarray}
Define the $k^{th}$ order initial layer as
\begin{eqnarray}\label{classical temp 24.}
\left\{
\begin{array}{rcl}
\ubc_{I,k}(\tau,\vx,\vw)&=&\ff_k(\tau,\vx,\vw)-\ff_k(\infty,\vx)\\
\p_{\tau}\bar\ff_k&=&-\displaystyle\int_{\s^1}\bigg(\vw\cdot\nabla_x\ubc_{I,k-1}\bigg)\ud{\vw},\\\rule{0ex}{1.5em}
\ff_k(\tau,\vx,\vw)&=&\ue^{-\tau}\ff_k(0,\vx,\vw)+\displaystyle\int_0^{\tau}\bigg(\bar\ff_k-\vw\cdot\nabla_x\ubc_{I,k-1}\bigg)(s,\vx,\vw)\ue^{s-\tau}\ud{s},\\
\ff_k(0,\vx,\vw)&=&\vw\cdot\nx\u_{k-1}(0,\vx,\vw)+\dt\u_{k-2}(0,\vx,\vw),\\
\lim_{\tau\rt\infty}\ff_k(\tau,\vx,\vw)&=&\ff_k(\infty,\vx).
\end{array}
\right.
\end{eqnarray}
Define the $k^{th}$ order interior solution as
\begin{eqnarray}
\left\{
\begin{array}{rcl}
\uc_k&=&\buc_k-\vw\cdot\nx\uc_{k-1}-\dt\uc_{k-2},\\\rule{0ex}{1em}
\dt\buc_k-\Delta_x\buc_k&=&0,\\\rule{0ex}{1em}\buc_k(0,\vx)&=&\ff_k(\infty,\vx)\ \
\text{in}\ \ \Omega,\\\rule{0ex}{1em} \buc_k&=&f_k(t,\infty,\theta)\
\ \text{on}\ \ \p\Omega.
\end{array}
\right.
\end{eqnarray}
By the idea in \cite{Bensoussan.Lions.Papanicolaou1979}, we should
be able to prove the following result:
\begin{theorem}\label{main fake 1.}
Assume $g(t,\vx_0,\vw)$ and $h(\vx,\vw)$ are sufficiently smooth. Then
for the unsteady neutron transport equation (\ref{transport}), the
unique solution $u^{\e}(t,\vx,\vw)\in
L^{\infty}([0,\infty)\times\Omega\times\s^1)$ satisfies
\begin{eqnarray}\label{main fake theorem 1.}
\lnm{u^{\e}-\uc_0-\ubc_{I,0}-\ubc_{B,0}}=O(\e).
\end{eqnarray}
\end{theorem}
\ \\
Similar to the analysis in \cite[Section 2.2]{AA003}, considering a
crucial observation that based on Remark \ref{Milne remark}, we know
that the existence of solution $\f_1$ requires
\begin{eqnarray}
\frac{\p
}{\p\theta}\bigg(\f_0(t,\eta,\theta,\xi)-f_0(t,\infty,\theta)\bigg)\in
L^{\infty}([0,\infty)^2\times[-\pi,\pi)\times[-\pi,\pi)).
\end{eqnarray}
This in turn requires
\begin{eqnarray}
\frac{\p \f_0}{\p\eta}\in
L^{\infty}([0,\infty)^2\times[-\pi,\pi)\times[-\pi,\pi)).
\end{eqnarray}
On the other hand, as shown by the Appendix of \cite{AA003}, we can
show for specific $g$, it holds that $\px\f_0\notin
L^{\infty}([0,\infty)^2\times[-\pi,\pi)\times[-\pi,\pi))$. Due to
intrinsic singularity for (\ref{classical temp 1.}), this
construction breaks down.

\subsection{Counterexample to Classical Approach}

\begin{theorem}
If $g(t,\theta,\phi)=t^2\ue^{-t}\cos\phi$ and $h(\vx,\vw)=0$, then there exists
a $C>0$ such that
\begin{eqnarray}
\lnm{u^{\e}-\uc_0-\ubc_{I,0}-\ubc_{B,0}}\geq C>0
\end{eqnarray}
when $\e$ is sufficiently small, where the interior solution $\uc_0$ is
defined in (\ref{classical temp 2.}), the initial layer $\ubc_{I,0}$ is defined in (\ref{classical temp 21.}), and the boundary layer $\ub_{B,0}$ is defined in (\ref{classical temp 1.}).
\end{theorem}
\begin{proof}
We divide the proof into several steps:\\
\ \\
Step 1: Basic settings.\\
By (\ref{classical temp 1.}), the solution $\f_0$ satisfies the Milne
problem
\begin{eqnarray}
\left\{
\begin{array}{rcl}\displaystyle
\sin(\theta+\xi)\frac{\p \f_0}{\p\eta}+\f_0-\bar \f_0&=&0,\\
\f_0(t,0,\theta,\xi)&=&g(t,\theta,\xi)\ \ \text{for}\ \
\sin(\theta+\xi)>0,\\\rule{0ex}{1em}
\lim_{\eta\rt\infty}\f_0(t,\eta,\theta,\xi)&=&f_0(t,\infty,\theta).
\end{array}
\right.
\end{eqnarray}
For convenience of comparison, we make the substitution
$\phi=\theta+\xi$ to obtain
\begin{eqnarray}
\left\{
\begin{array}{rcl}\displaystyle
\sin\phi\frac{\p \f_0}{\p\eta}+\f_0-\bar \f_0&=&0,\\
\f_0(t,0,\theta,\phi)&=&g(t,\theta,\phi)\ \ \text{for}\ \
\sin\phi>0,\\\rule{0ex}{1em}
\lim_{\eta\rt\infty}\f_0(t,\eta,\theta,\phi)&=&f_0(t,\infty,\theta).
\end{array}
\right.
\end{eqnarray}
Assume the theorem is incorrect, i.e.
\begin{eqnarray}
\lim_{\e\rt0}\lnm{(\uc_0+\ubc_{I,0}+\ubc_{B,0})-(\u_0+\ub_{I,0}+\ub_{B,0})}=0.
\end{eqnarray}
We can easily show the zeroth order initial layer $\ubc_{B,0}=\ub_{B,0}=0$ due to $h(\vx,\vw)=0$. Since the boundary $g(t,\theta,\phi)=t^2\ue^{-t}\cos\phi$ independent of
$\theta$, by (\ref{classical temp 1.}) and (\ref{expansion temp 9}),
it is obvious the limit of zeroth order boundary layer
$f_0(t,\infty,\theta)$ and $f_0^{\e}(t,\infty,\theta)$ satisfy
$f_0(t,\infty,\theta)=C_1(t)$ and $f_0^{\e}(t,\infty,\theta)=C_2(t)$
for some constant $C_1(t)$ and $C_2(t)$ independent of $\theta$. By
(\ref{classical temp 2.}), (\ref{expansion temp 8}) and solution continuity of heat equation, we can derive
the interior solutions are smooth and are close to constants
$\uc_0=C_1(t)$ and $\u_0=C_2(t)$ in a neighborhood $O(\e)$ of the
boundary with difference $O(\e)$. Hence, we may further derive in
this neighborhood,
\begin{eqnarray}\label{compare temp 5.}
\lim_{\e\rt0}\lnm{(f_0(\infty)+\ubc_0)-(f_0^{\e}(\infty)+\ub_0)}=0.
\end{eqnarray}
For $0\leq\eta\leq 1/(2\e)$, we have $\psi_0=1$, which means
$\f_0=\ubc_0+f_0(\infty)$ and $f_0^{\e}=\ub_0+f_0^{\e}(\infty)$ in this neighborhood of the boundary. Define $u=f_0+2$, $U=f_0^{\e}+2$ and
$G=g+2=t^2\ue^{-t}\cos\phi+2$, then $u(\eta,\phi)$ satisfies the equation
\begin{eqnarray}\label{compare flat equation.}
\left\{
\begin{array}{rcl}\displaystyle
\sin\phi\frac{\p u}{\p\eta}+u-\bar u&=&0,\\
u(0,\phi)&=&G(\phi)\ \ \text{for}\ \ \sin\phi>0,\\\rule{0ex}{1em}
\lim_{\eta\rt\infty}u(\eta,\phi)&=&2+f_0(\infty),
\end{array}
\right.
\end{eqnarray}
and $U(\eta,\phi)$ satisfies the equation
\begin{eqnarray}\label{compare force equation.}
\left\{
\begin{array}{rcl}\displaystyle
\sin\phi\frac{\p U}{\p\eta}+F(\e;\eta)\cos\phi \frac{\p
U}{\p\phi}+U-\bar
U&=&0,\\
U(0,\phi)&=&G(\phi)\ \ \text{for}\ \ \sin\phi>0,\\\rule{0ex}{1em}
\lim_{\eta\rt\infty}U(\eta,\phi)&=&2+f_0^{\e}(\infty).
\end{array}
\right.
\end{eqnarray}
Based on (\ref{compare temp 5.}), we have
\begin{eqnarray}
\lim_{\e\rt0}\lnm{U(\eta,\phi)-u(\eta,\phi)}=0.
\end{eqnarray}
Then it naturally implies
\begin{eqnarray}
\lim_{\e\rt0}\lnm{\bar U(\eta)-\bar u(\eta)}=0.
\end{eqnarray}
\ \\
Step 2: Continuity of $\bar u$ and $\bar U$ at $\eta=0$.\\
For the problem (\ref{compare flat equation.}), we have for any
$r_0>0$
\begin{eqnarray}
\abs{\bar u(\eta)-\bar
u(0)}&\leq&\frac{1}{2\pi}\bigg(\int_{\sin\phi\leq
r_0}\abs{u(\eta,\phi)-u(0,\phi)}\ud{\phi}+\int_{\sin\phi\geq
r_0}\abs{u(\eta,\phi)-u(0,\phi)}\ud{\phi}\bigg).
\end{eqnarray}
Since we have shown $u\in L^{\infty}([0,\infty)\times[-\pi,\pi))$,
then for any $\delta>0$, we can take $r_0$ sufficiently small such
that
\begin{eqnarray}
\frac{1}{2\pi}\int_{\sin\phi\leq
r_0}\abs{u(\eta,\phi)-u(0,\phi)}\ud{\phi}&\leq&\frac{C}{2\pi}\arcsin
r_0\leq \frac{\delta}{2}.
\end{eqnarray}
For fixed $r_0$ satisfying above requirement, we estimate the
integral on $\sin\phi\geq r_0$. By Ukai's trace theorem, $u(0,\phi)$
is well-defined in the domain $\sin\phi\geq r_0$ and is continuous.
Also, by consider the relation
\begin{eqnarray}
\frac{\p u}{\p\eta}(0,\phi)=\frac{\bar u(0)-u(0,\phi)}{\sin\phi},
\end{eqnarray}
we can obtain in this domain $\px u$ is bounded, which further
implies $u(\eta,\phi)$ is uniformly continuous at $\eta=0$. Then
there exists $\delta_0>0$ sufficiently small, such that for any
$0\leq\eta\leq\delta_0$, we have
\begin{eqnarray}
\frac{1}{2\pi}\int_{\sin\phi\geq
r_0}\abs{u(\eta,\phi)-u(0,\phi)}\ud{\phi}&\leq&\frac{1}{2\pi}\int_{\sin\phi\geq
r_0}\frac{\delta}{2}\ud{\phi}\leq\frac{\delta}{2}.
\end{eqnarray}
In summary, we have shown for any $\delta>0$, there exists
$\delta_0>0$ such that for any $0\leq\eta\leq\delta_0$,
\begin{eqnarray}
\abs{\bar u(\eta)-\bar
u(0)}\leq\frac{\delta}{2}+\frac{\delta}{2}=\delta.
\end{eqnarray}
Hence, $\bar u(\eta)$ is continuous at $\eta=0$. By a similar
argument along the characteristics, we can show $\bar U(\eta,\phi)$
is also continuous at $\eta=0$.

In the following, by the continuity, we assume for arbitrary
$\delta>0$, there exists a $\delta_0>0$ such that for any
$0\leq\eta\leq\delta_0$, we have
\begin{eqnarray}
\abs{\bar u(\eta)-\bar u(0)}&\leq&\delta\label{compare temp 1.},\\
\abs{\bar U(\eta)-\bar U(0)}&\leq&\delta\label{compare temp 2.}.
\end{eqnarray}
\ \\
Step 3: Milne formulation.\\
We consider the solution at a specific point $(\eta,\phi)=(n\e,\e)$
for some fixed $n>0$. The solution along the characteristics can be
rewritten as follows:
\begin{eqnarray}\label{compare temp 3.}
u(n\e,\e)=G(\e)\ue^{-\frac{1}{\sin\e}n\e}
+\int_0^{n\e}\ue^{-\frac{1}{\sin\e}(n\e-\k)}\frac{1}{\sin\e}\bar
u(\k)\ud{\k},
\end{eqnarray}
\begin{eqnarray}\label{compare temp 4.}
U(n\e,\e)=G(\e_0)\ue^{-\int_0^{n\e}\frac{1}{\sin\phi(\zeta)}\ud{\zeta}}
+\int_0^{n\e}\ue^{-\int_{\k}^{n\e}\frac{1}{\sin\phi(\zeta)}\ud{\zeta}}\frac{1}{\sin\phi(\k)}\bar
U(\k)\ud{\k},
\end{eqnarray}
where we have the conserved energy along the characteristics
\begin{eqnarray}
E(\eta,\phi)=\cos\phi \ue^{-V(\eta)},
\end{eqnarray}
in which $(0,\e_0)$ and $(\zeta,\phi(\zeta))$ are in the same
characteristics of $(n\e,\e)$.\\
\ \\
Step 4: Estimates of (\ref{compare temp 3.}).\\
We turn to the Milne problem for $u$. We have the natural estimate
\begin{eqnarray}
\int_0^{n\e}\ue^{-\frac{1}{\sin\e}(n\e-\k)}\frac{1}{\sin\e}\ud{\k}&=&\int_0^{n\e}\ue^{-\frac{1}{\e}(n\e-\k)}\frac{1}{\e}\ud{\k}+o(\e)\\
&=&\ue^{-n}\int_0^{n\e}\ue^{\frac{\k}{\e}}\frac{1}{\e}\ud{\k}+o(\e)\nonumber\\
&=&\ue^{-n}\int_0^n\ue^{\zeta}\ud{\zeta}+o(\e)\nonumber\\
&=&(1-\ue^{-n})+o(\e)\nonumber.
\end{eqnarray}
Then for $0<\e\leq\delta_0$, we have $\abs{\bar u(0)-\bar
u(\k)}\leq\delta$, which implies
\begin{eqnarray}
\int_0^{n\e}\ue^{-\frac{1}{\sin\e}(n\e-\k)}\frac{1}{\sin\e}\bar
u(\k)\ud{\k}&=&
\int_0^{n\e}\ue^{-\frac{1}{\sin\e}(n\e-\k)}\frac{1}{\sin\e}\bar u(0)\ud{\k}+O(\delta)\\
&=&(1-\ue^{-n})\bar u(0)+o(\e)+O(\delta)\nonumber.
\end{eqnarray}
For the boundary data term, it is easy to see
\begin{eqnarray}
G(\e)\ue^{-\frac{1}{\sin\e}n\e}&=&\ue^{-n}G(\e)+o(\e)
\end{eqnarray}
In summary, we have
\begin{eqnarray}
u(n\e,\e)=(1-\ue^{-n})\bar u(0)+\ue^{-n}G(\e)+o(\e)+O(\delta).
\end{eqnarray}
\ \\
Step 5: Estimates of (\ref{compare temp 4.}).\\
We consider the $\e$-Milne problem for $U$. For $\e<<1$ sufficiently
small, $\psi(\e)=1$. Then we may estimate
\begin{eqnarray}
\cos\phi(\zeta)\ue^{-V(\zeta)}=\cos\e \ue^{-V(n\e)},
\end{eqnarray}
which implies
\begin{eqnarray}
\cos\phi(\zeta)=\frac{1-n\e^2}{1-\e\zeta}\cos\e.
\end{eqnarray}
and hence
\begin{eqnarray}
\sin\phi(\zeta)=\sqrt{1-\cos^2\phi(\zeta)}=\sqrt{\frac{\e(n\e-\zeta)(2-\e\zeta-n\e^2)}{(1-\e\zeta)^2}\cos^2\e+\sin^2\e}.
\end{eqnarray}
For $\zeta\in[0,\e]$ and $n\e$ sufficiently small, by Taylor's
expansion, we have
\begin{eqnarray}
1-\e\zeta&=&1+o(\e),\\
2-\e\zeta-n\e^2&=&2+o(\e),\\
\sin^2\e&=&\e^2+o(\e^3),\\
\cos^2\e&=&1-\e^2+o(\e^3).
\end{eqnarray}
Hence, we have
\begin{eqnarray}
\sin\phi(\zeta)=\sqrt{\e(\e+2n\e-2\zeta)}+o(\e^2).
\end{eqnarray}
Since $\sqrt{\e(\e+2n\e-2\zeta)}=O(\e)$, we can further estimate
\begin{eqnarray}
\frac{1}{\sin\phi(\zeta)}&=&\frac{1}{\sqrt{\e(\e+2n\e-2\zeta)}}+o(1)\\
-\int_{\k}^{n\e}\frac{1}{\sin\phi(\zeta)}\ud{\zeta}&=&\sqrt{\frac{\e+2n\e-2\zeta}{\e}}\bigg|_{\k}^{n\e}+o(\e)
=1-\sqrt{\frac{\e+2n\e-2\k}{\e}}+o(\e).
\end{eqnarray}
Then we can easily derive the integral estimate
\begin{eqnarray}
\int_0^{n\e}\ue^{-\int_{\k}^{n\e}\frac{1}{\sin\phi(\zeta)}\ud{\zeta}}\frac{1}{\sin\phi(\k)}\ud{\k}&=&
\ue^1\int_0^{n\e}\ue^{-\sqrt{\frac{\e+2n\e-2\k}{\e}}}\frac{1}{\sqrt{\e(\e+2n\e-2\k)}}\ud{\k}+o(\e)\\
&=&\half \ue^1\int_{\e}^{(1+2n)\e}\ue^{-\sqrt{\frac{\sigma}{\e}}}\frac{1}{\sqrt{\e\sigma}}\ud{\sigma}+o(\e)\nonumber\\
&=&\half \ue^1\int_{1}^{1+2n}\ue^{-\sqrt{\rho}}\frac{1}{\sqrt{\rho}}\ud{\rho}+o(\e)\nonumber\\
&=&\ue^1\int_{1}^{\sqrt{{1+2n}}}\ue^{-t}\ud{t}+o(\e)\nonumber\\
&=&(1-\ue^{1-\sqrt{1+2n}})+o(\e)\nonumber.
\end{eqnarray}
Then for $0<\e\leq\delta_0$, we have $\abs{\bar U(0)-\bar
U(\k)}\leq\delta$, which implies
\begin{eqnarray}
\int_0^{n\e}\ue^{-\int_{\k}^{n\e}\frac{1}{\sin\phi(\zeta)}\ud{\zeta}}\frac{1}{\sin\phi(\k)}\bar
U(\k)\ud{\k}&=&
\int_0^{n\e}\ue^{-\int_{\k}^{n\e}\frac{1}{\sin\phi(\zeta)}\ud{\zeta}}\frac{1}{\sin\phi(\k)}\bar U(0)\ud{\k}+O(\delta)\\
&=&(1-\ue^{1-\sqrt{1+2n}})\bar U(0)+o(\e)+O(\delta)\nonumber.
\end{eqnarray}
For the boundary data term, since $G(\phi)$ is $C^1$, a similar
argument shows
\begin{eqnarray}
G(\e_0)\ue^{-\int_0^{n\e}\frac{1}{\sin\phi(\zeta)}\ud{\zeta}}&=&\ue^{1-\sqrt{1+2n}}G(\sqrt{1+2n}\e)+o(\e).
\end{eqnarray}
Therefore, we have
\begin{eqnarray}
U(n\e,\e)=(1-\ue^{1-\sqrt{1+2n}})\bar
U(0)+\ue^{1-\sqrt{1+2n}}G(\sqrt{1+2n}\e)+o(\e)+O(\delta).
\end{eqnarray}
\ \\
Step 6: Contradiction.\\
In summary, we have the estimate
\begin{eqnarray}
u(n\e,\e)&=&(1-\ue^{-n})\bar u(0)+\ue^{-n}G(\e)+o(\e)+O(\delta),\\
U(n\e,\e)&=&(1-\ue^{1-\sqrt{1+2n}})\bar
U(0)+\ue^{1-\sqrt{1+2n}}G(\sqrt{1+2n}\e)+o(\e)+O(\delta).
\end{eqnarray}
The boundary data is $G=t^2\ue^{-t}\cos\phi+2$. Fix $t=1$. Then by the maximum principle
in Theorem \ref{Milne theorem 3}, we can achieve $1\leq
u(0,\phi)\leq3$ and $1\leq U(0,\phi)\leq3$. Since
\begin{eqnarray}
\bar u(0)&=&\frac{1}{2\pi}\int_{-\pi}^{\pi}u(0,\phi)\ud{\phi}
=\frac{1}{2\pi}\int_{\sin\phi>0}u(0,\phi)\ud{\phi}+\frac{1}{2\pi}\int_{\sin\phi<0}u(0,\phi)\ud{\phi}\\
&=&\frac{1}{2\pi}\int_{\sin\phi>0}(2+\ue^{-1}\cos\phi)\ud{\phi}+\frac{1}{2\pi}\int_{\sin\phi<0}u(0,\phi)\ud{\phi}\nonumber\\
&=&2+\frac{1}{2\pi}\int_{\sin\phi>0}\ue^{-1}\cos\phi\ud{\phi}+\frac{1}{2\pi}\int_{\sin\phi<0}u(0,\phi)\ud{\phi}\nonumber,
\end{eqnarray}
we naturally obtain
\begin{eqnarray}
2-\half\ue^{-1}\leq\bar u(0)\leq 2+\half\ue^{-1}.
\end{eqnarray}
Similarly, we can
obtain
\begin{eqnarray}
2-\half\ue^{-1}\leq\bar U(0)\leq 2+\half\ue^{-1}.
\end{eqnarray}
Furthermore, for $\e$
sufficiently small, we have
\begin{eqnarray}
G(\sqrt{1+2n}\e)&=&2+\ue^{-1}+o(\e),\\
G(\e)&=&2+\ue^{-1}+o(\e).
\end{eqnarray}
Hence, we can obtain
\begin{eqnarray}
u(n\e,\e)&=&\bar u(0)+\ue^{-n}(-\bar u(0)+2+\ue^{-1})+o(\e)+O(\delta),\\
U(n\e,\e)&=&\bar U(0)+\ue^{1-\sqrt{1+2n}}(-\bar
U(0)+2+\ue^{-1})+o(\e)+O(\delta).
\end{eqnarray}
Then we can see $\lim_{\e\rt0}\lnm{\bar U(0)-\bar u(0)}=0$ naturally
leads to $\lim_{\e\rt0}\lnm{(-\bar u(0)+2+\ue^{-1})-(-\bar U(0)+2+\ue^{-1})}=0$. Also,
we have $-\bar u(0)+2+\ue^{-1}=O(1)$ and $-\bar U(0)+2+\ue^{-1}=O(1)$. Due to the
smallness of $\e$ and $\delta$, and also $\ue^{-n}\neq
\ue^{1-\sqrt{1+2n}}$, we can obtain
\begin{eqnarray}
\abs{U(n\e,\e)-u(n\e,\e)}=O(1).
\end{eqnarray}
However, above result contradicts our assumption that
$\lim_{\e\rt0}\lnm{U(\eta,\phi)-u(\eta,\phi)}=0$ for any
$(\eta,\phi)$. This completes the proof.
\end{proof}

\section*{Acknowledgements}

The author thanks Yan Guo and Xiongfeng Yang for stimulating
discussions. The research is supported by NSF grant
0967140.

\bibliographystyle{siam}
\bibliography{Reference}

\end{document}